\documentclass[12 pt,reqno]{amsart}
\usepackage[T1]{fontenc}  
\usepackage{lmodern}
\usepackage{ upgreek }
\usepackage{latexsym, mathrsfs, amsmath, bm, amssymb, longtable, booktabs,amscd,microtype,booktabs,cases}
\usepackage[square,numbers]{natbib}
\usepackage{graphicx}
\usepackage[dvipsnames]{xcolor}
\usepackage{caption}
\usepackage{dsfont}
\usepackage[all]{xy}
\usepackage{enumerate}
\usepackage{amsthm}
\usepackage{multirow}
\usepackage{tikz}
\usetikzlibrary{matrix,arrows,decorations.pathmorphing}
\usepackage{collectbox}
\usepackage{enumerate}
\usepackage[colorinlistoftodos,prependcaption, textsize=tiny]{todonotes}
\reversemarginpar
\usepackage{amsmath}
\usepackage{enumitem}
\usepackage{csquotes}
\usepackage{hyperref}
\let\oldtocsection=\tocsection

\let\oldtocsubsection=\tocsubsection

\let\oldtocsubsubsection=\tocsubsubsection

\renewcommand{\tocsection}[2]{\hspace{0em}{\vspace{0.5em}}\oldtocsection{#1}{#2}}

\renewcommand{\tocsubsection}[2]{\hspace{1em}{\vspace{0.5em}}\oldtocsubsection{#1}{#2}}
\renewcommand{\tocsubsubsection}[2]{\hspace{2em}\oldtocsubsubsection{#1}{#2}}
\hypersetup{
	colorlinks,
	citecolor=red,
	filecolor=black,
	linkcolor=blue,
	urlcolor=black
}

\numberwithin{equation}{section} 
\usepackage{mathtools}

\newcommand{\Z}{\mathbb{Z}}

\newcommand{\C}{\mathbb{C}}

\newcommand{\B}{\mathcal{B}}

\makeatother

\DeclareMathOperator{\ord}{ord}
\DeclareMathOperator{\pideg}{PI-deg}
\DeclareMathOperator{\dime}{dim}
\DeclareMathOperator{\ran}{rank}

\DeclareMathOperator{\tors}{tor}

\newtheorem{thm}{Theorem}[section]
\newtheorem{theorem}[thm]{Theorem}
\newtheorem{cor}[thm]{Corollary}

\newtheorem{prop}[thm]{Proposition}
\newtheorem{lemma}[thm]{Lemma}
\theoremstyle{definition}

\newtheorem{remark}[thm]{Remark}

\theoremstyle{definition}
\newtheorem{dfn}{Definition}

\theoremstyle{remark}

\theoremstyle{remark}

\makeatletter
\def\imod#1{\allowbreak\mkern10mu({\operator@font mod}\,\,#1)}
\makeatother

\begin{document}

\title{Simple Modules and PI Structure of the Two-Parameter Quantized Algebra $ U^+_{r,s}(B_2) $}
\author[ S. Mukherjee, \ Ritesh Kumar Pandey]{Snehashis Mukherjee,  \ Ritesh Kumar Pandey}
\address {\newline Snehashis Mukherjee$^1$ \newline Indian Institute of Technology Kanpur, 
  Kalyanpur, Kanpur, Uttar Pradesh, Box: 208016, India.
  \newline Ritesh Kumar Pandey$^2$ \newline Indian Institute of Technology Kanpur, 
  Kalyanpur, Kanpur, Uttar Pradesh, Box: 208016, India.}
\email{\href{mailto:snehashism@iitk.ac.in}{snehashism@iitk.ac.in$^1$};\href{mailto:riteshp@iitk.ac.in}{riteshp@iitk.ac.in$^2$}}

\begin{abstract}
We study the two-parameter quantized enveloping algebra $ U^+_{r,s}(B_2) $ at  roots of unity and investigate its structure and representations. We first show that when $ r $ and $ s $ are roots of unity, the algebra becomes a PI algebra, and we compute its PI degree explicitly using the De Concini–Procesi method. We construct and classify finite dimensional simple modules for $ U^+_{r,s}(B_2) $ by analyzing a subalgebra $ \mathcal{B} \subset U^+_{r,s}(B_2) $. Simple modules are categorized into torsion-free and torsion types with respect to a distinguished normal element. We classify all torsion-free simple $ \B $-modules and lift them to $ U^+_{r,s}(B_2) $. The remaining simple modules are constructed in the nilpotent case. This work provides a complete classification of simple $ U^+_{r,s}(B_2) $-modules at roots of unity and contributes to the understanding of two-parameter quantum groups in type $ B_2 $.
\end{abstract}
\keywords{Two parameter quantized algebra, Polynomial Identity algebra, PI degree, Simple modules, Indecomposable modules}
\subjclass{16D60, 16D70, 16R20, 16T20, 16S85}
\maketitle

{\bf{Notations and Assumptions:}} 
\begin{itemize}
\item The sets of complex numbers, integers, and non-negative integers are denoted by $\C$, $\Z$, and   $\Z_{\geq 0}$, respectively.
\item For any non-zero integer $m\in\Z$, let $e_2(m)$ denote the exponent of $2$ in the prime factorization of $m$.
\item In this paper, a module always means a right module.
\item $r$ and $s$ are primitive $m$-th and $n$-th root of unity, respectively, with $r^2\neq s^2$.
\end{itemize}

\section{Introduction}

The theory of quantum groups, initiated by Drinfel’d and Jimbo in the 1980s, has profoundly influenced modern mathematics and theoretical physics. These noncommutative deformations of universal enveloping algebras associated to semisimple Lie algebras give rise to rich algebraic structures, intertwining representation theory, and categorification. In recent decades, the development of \emph{multi-parameter} quantum groups has introduced new algebraic perspectives, with two-parameter versions offering a significant generalization of the classical one-parameter theory.
\par The two-parameter quantized enveloping algebras $ U_{r,s}(\mathfrak{g}) $, initially introduced systematically by Benkart and Witherspoon for type $ A $ \cite{sb2,sb1},  Bergeron et al. \cite{ber}, and Hu et al. \cite{hu1,hu4,hu,hu3,hu2} later extended to other types including $B, C, D$, and $ G_2 $, possess structural and representational differences from their one-parameter counterparts. These algebras are generally non-isomorphic to the standard (one-parameter) quantum groups. Although they share some formal similarities, they typically display less symmetry and a more rigid automorphism structure \cite{Tang1}. In particular, the positive part $ U^+_{r,s}(B_2) $---the focus of this work---serves as a key testing ground for understanding these intricacies in rank-2 non-simply-laced types. Recently, there has been growing interest in the study of subalgebras of two-parameter quantized enveloping algebras. Bera et al.~\cite{sb} studied the simple modules over $U_{r,s}^+(\mathfrak{sl}_3)$ and provided a complete classification assuming that $r$ and $s$ are roots of unity. They first established that under this condition, the algebra becomes a Polynomial Identity (PI) algebra and determined its PI degree to be $\operatorname{lcm}(r,s)$. Additionally, they examined when the classified simple modules achieve the maximal dimension permitted by the PI degree.

\par The algebra $ U^+_{r,s}(B_2) $, generated by quantum root vectors $ e_1, e_2 $ and defined by two $(r, s)$-deformed Serre-type relations, can be viewed as an iterated skew polynomial ring. It admits a PBW-type Lyndon basis arising from a convex ordering on the positive root system \cite{huwang}. This basis enables a systematic study of the algebra's internal structure and plays a vital role in the analysis of its representations. In particular, when the parameters $r$ and $ s $ are roots of unity, $ U^+_{r,s}(B_2) $ becomes a PI algebra, allowing the use of powerful structural theorems from PI theory and facilitating the classification of its finite dimensional simple modules.
\par Much of the existing work on two-parameter algebras has emphasized their ring-theoretic and homological properties. For instance, the derivations, Hochschild cohomology, and automorphism groups of $ U^+_{r,s}(B_2) $ have been computed explicitly using embeddings into quantum tori \cite{Tang1}. Moreover, the spectrum of prime and primitive ideals has been stratified using Goodearl--Letzter theory, and the Dixmier--Moeglin equivalence has been established for these algebras under mild transcendentality assumptions \cite{Tang2}. Despite this progress, the fine structure of simple modules for $ U^+_{r,s}(B_2) $, especially at roots of unity, remains an area requiring further exploration. Our work fills this gap by providing a complete classification of finite dimensional simple $ U^+_{r,s}(B_2) $-modules when $ r $ and $ s $ are roots of unity.
\par This paper builds on the earlier work of Bera and the first author on the one-parameter quantum algebra $ U^+_q(B_2) $ at roots of unity \cite{sbsm}. There, it was proved that the algebra is PI. Additionally, they determined its PI degree and center, as well as constructed all of its simple modules using a tower of generalized Weyl subalgebras. The present work generalizes these results to the two-parameter setting. While the techniques are similar in spirit—such as reducing to a more tractable subalgebra and using torsion theory—the two-parameter context introduces significant new challenges due to the non-symmetric parameter dependence and more intricate commutation relations.
\par In this article, we investigate the representation theory of $ U^+_{r,s}(B_2) $ when the parameters $ r, s \in \mathbb{C}^* $ are roots of unity. This setting is of particular interest due to the emergence of nontrivial centers, large families of finite dimensional simple modules, and links to the theory of quantum finite dimensional Hopf algebras. Our main contributions are twofold. First, we provide a detailed analysis of the algebra's PI properties in the root of unity case, including a computation of its PI degree. This is achieved by expressing $ U^+_{r,s}(B_2) $ as an iterated Ore extension, calculating the corresponding multiplicatively skew-symmetric matrix, and applying the De Concini--Procesi method. Second, we construct and classify several families of simple modules by exploiting torsion phenomena with respect to normal elements and by introducing an auxiliary subalgebra $ \B \subset U^+_{r,s}(B_2) $ that serves as a simpler model retaining essential features of the larger algebra. 
\par Launois et al.~\cite{lau2} extended both the deleting derivations algorithm and Cauchon's canonical embedding to a broad class of quantum nilpotent algebras at roots of unity. These algebras are defined as iterated Ore extensions equipped with automorphisms and derivations that satisfy specific compatibility conditions. Their method constructs a quantum affine space $\mathcal{A}$ from a given quantum algebra $A$ via a sequence of variable transformations within the division ring of fractions $\mathrm{Frac}(A)$. The canonical embedding maps a completely prime ideal $P$ of $A$ to a completely prime ideal $Q$ of $\mathcal{A}$, preserving PI degrees in the case where $A$ is a PI algebra, i.e.,
\begin{center}$
\operatorname{PIdeg}(A/P) = \operatorname{PIdeg}(\mathcal{A}/Q).
$\end{center}
Furthermore, they developed a procedure for constructing a maximal-dimensional irreducible representation of $A/P$, assuming one has access to an appropriate irreducible representation of $\mathcal{A}/Q$. As an illustration of this approach, they provided an explicit irreducible representation of the completely prime quotient $U_q^+(B_2)/\langle z' \rangle$, where $z'$ is a central element of $U_q^+(B_2)$~\cite[Example 5.1]{lau2}. This technique offers a valuable tool for studying a particular subclass of simple modules, namely those of maximal dimension. However, it does not yield all such simple modules over the algebra $A$. Additionally, the construction assumes that certain generators of $A$ act invertibly on the corresponding simple module, which may limit its applicability in more general settings. It is also worth noting that the automorphisms and derivations arising in the iterated skew polynomial presentation of our algebra do not satisfy the compatibility conditions required by Launois et al. As a result, our algebra does not qualify as a quantum nilpotent algebra in their framework.
\par A key strategy in our work is to reduce the complexity of the full algebra $ U^+_{r,s}(B_2) $ by focusing on a well-chosen subalgebra $ \B $ that admits a GWA (generalized Weyl algebra) structure. This allows us to analyze modules via torsion with respect to central or normal elements and apply known classification techniques from noncommutative algebra. In particular, we identify and construct simple torsion-free modules for $ \B $, and then lift these to modules for $ U^+_{r,s}(B_2) $ via extension and localization techniques. This approach parallels and extends earlier work done in rank-1 and simply-laced settings, adapting it to the more intricate two-parameter, rank-2, non-simply-laced context.
\par Our results also contribute to the ongoing effort to understand how two-parameter deformations behave in comparison to their one-parameter analogues. While the algebra $ U^+_{q,q^{-1}}(B_2) $ (with $ r = q, s = q^{-1} $) can be recovered as a specialization, the generic two-parameter case shows considerable deviation in both structure and representation theory, particularly in the presence of torsion phenomena and in the breakdown of Lusztig symmetries \cite{huwang}.

\par  Throughout the article we classify $X_3$-torsionfree simple $ U^+_{r,s}(B_2) $-modules. If $\mathcal{N}$ is a $X_3$-torsion simple $ U^+_{r,s}(B_2) $-module, then $\mathcal{N}$ becomes a module over $ U^+_{r,s}(B_2)/\langle X_3\rangle \cong U_{r,s}^+(\mathfrak{sl}_3) $. All simple $U_{r,s}^+(\mathfrak{sl}_3)$-modules at roots of unity have already been classified in \cite{sb}.

This paper is organized as follows. In {Section 2}, we define the algebra $ U^+_{r,s}(B_2) $ and give a PBW-type basis for it with respect to the newly introduced elements. We then provide a skew polynomial presentation of $U^+_{r,s}(B_2)$. Finally, we prove in Theorem \ref{m1} that the algebra $U^+_{r,s}(B_2)$ is a PI algebra if and only if $r$ and $s$ are roots of unity. We devote {Section 3} to computing the PI degree of $U^+_{r,s}(B_2)$ at roots of unity (see Theorem \ref{m2}). In {Section 4}, we introduce a distinguished subalgebra $\mathcal{B} \subseteq U^+_{r,s}(B_2)$, which retains the essential structure of $U^+_{r,s}(B_2)$ but is algebraically simpler to handle. This reduction plays a crucial role in the latter classification. We prove a one-to-one correspondence between $X_1$-torsionfree simple $\mathcal{B}$-modules and $X_1$-torsionfree simple $U^+_{r,s}(B_2)$-modules (see Theorem \ref{itd}). In Sections 5 and  6, we explicitly construct and classify simple modules over $\B$, focusing on torsion-free modules with respect to a normal element $ X_1 $. For each type, we define an explicit basis and module action, and we prove simplicity using eigenvector and length reduction arguments (see Theorems \ref{dim1}-\ref{dim3}).  In {Section 7}, we use Theorem \ref{itd} to classify all simple $X_1$-torsionfree $ U^+_{r,s}(B_2) $-modules. Sections 8 and 9 address the remaining class of simple $ U^+_{r,s}(B_2) $-modules, namely those for which $ X_1 $ acts nilpotently (see Theorem \ref{m4}). Section 10 establishes necessary and sufficient conditions for two simple modules in the same class to be isomorphic, completing the classification problem (see Theorems \ref{1stIso} - \ref{5thIso}). Finally, in Section 11, we give a class of finite dimensional indecomposable $U_{r,s}^+(B_2)$-modules that are not simple. We end the section with a remark that opens up several future directions.

\section{Preliminaries} 
\subsection{\texorpdfstring {Two-parameter quantized enveloping algebra $U_{r,s}^+(B_2)$}{}}
The two-parameter quantized enveloping algebra $U^+_{r,s}(B_2) $ is an algebra over $\C$, generated by the elements $ e_1, e_2$ subject to the following defining relations:
\begin{align*}
    e_1^2 e_2 - (r^2 + s^2) e_1 e_2 e_1 + r^2 s^2 e_2 e_1^2 &= 0,\\
    e_1 e_2^3 - (r^2 + rs + s^2) e_2 e_1 e_2^2 + rs(r^2 + rs + s^2) e_2^2 e_1 e_2 - r^3 s^3 e_2^3 e_1 &= 0.\end{align*}

To simplify the expressions, we introduce the following elements \cite[Lemma 1.1]{Tang1}:
    \begin{align*}
    X_1 &:= e_1,\quad X_2 := e_1 e_2 - r^2 e_2 e_1,
     \\
    X_3 &:= e_2 X_2 - s^{-2} X_2 e_2,\quad X_4 := e_2.
    \end{align*}
    
It is easy to verify that these satisfy the following relations:
\begin{equation}\label{Eq2.1}
    \left.\begin{aligned}
    X_1 X_2 &= s^2 X_2 X_1,\quad X_1 X_3 = r^2 s^2 X_3 X_1,\\
X_1 X_4 &= r^2 X_4 X_1 + X_2, \quad X_2 X_3 = rs X_3 X_2, \\
    X_2 X_4 &= s^2 X_4 X_2 - s^2 X_3,\quad  X_3 X_4 = rs X_4 X_3.
\end{aligned}\right\}\end{equation}

Moreover, the set  
$$
    \{X_1^a X_2^b X_3^c X_4^d \mid a, b, c, d \in \mathbb{Z}_{\geq 0} \}
$$
forms a PBW type basis for $U^+_{r,s}(B_2) $ \cite[Corollay 1.1]{Tang1}.
\subsection{\texorpdfstring{Skew Polynomial Presentation of $ U_{r,s}^+(B_2) ${}}{}}\label{Sub2.2} The algebra $U^+_{r,s}(B_2)$ can be presented as an iterated skew polynomial ring:
\begin{equation}\label{Tang1}
U^+_{r,s}(B_2) \cong \mathbb{C}[X_1][X_2;\tau_2][X_3;\tau_3][X_4;\tau_4,\delta_4]
\end{equation}
where $\tau_2$, $\tau_3$, $\tau_4$ are algebra automorphisms of $\mathbb{C}[X_1], \mathbb{C}[X_1][X_2;\tau_2]$ and $\mathbb{C}[X_1][X_2;\tau_2][X_3;\tau_3]$ respectively and $\delta_4$  is a $\tau_4$-derivation of $\mathbb{C}[X_1][X_2;\tau_2][X_3;\tau_3]$ defined as follows:
\begin{align*}
\tau_2(X_1) &= s^{-2}X_1, &\tau_3(X_1) &= r^{-2}s^{-2}X_1, & \tau_3(X_2) &= r^{-1}s^{-1}X_2, \\
\tau_4(X_1) &= r^{-2}X_1, & \tau_4(X_2) &= s^{-2}X_2, & \tau_4(X_3) &= r^{-1}s^{-1}X_3, \\
\delta_4(X_1) &= -r^{-1}X_2, & \delta_4(X_2) &= X_3, & \delta_4(X_3) &= 0.
\end{align*}
\begin{remark} In the case where $\mathfrak{g}$ is of type $B_2$, the Weyl group $\mathcal{W}$ is isomorphic to the dihedral group $D_4$. Its longest element is $w_0 = -\mathrm{Id}$, which admits exactly two reduced expressions: $w_0 = s_1 s_2 s_1 s_2 = s_2 s_1 s_2 s_1$. In \cite[Definition 2.3]{huwang}, Hu and Wang selected the reduced decomposition $w_0 = (s_1 s_2 s_1) s_2$ to construct an alternative PBW basis for $U_{r,s}^+(B_2)$. This basis was then used to show that under a certain
condition, the restricted two-parameter quantum group $U_{r,s}^+(B_2)$  is a Drinfel’d double. The associated iterated Ore extension of $U_{r,s}^+(B_2)$, based on this PBW basis, involves two derivations in its construction; see \cite[Theorem 2.4]{huwang}.

In contrast, the iterated Ore extension introduced in Equation~(\ref{Tang1}) involves only a single derivation. This simplification proves particularly useful in the present work for analyzing the representation theory of $U_{r,s}^+(B_2)$ in the root of unity setting.
\end{remark}
 We have the following lemma.
\begin{lemma}\label{lemma2.2}
    The following identities hold in $ U_{r,s}^+(B_2) $.
    \begin{enumerate}[label= \((\arabic*)\)]
    \item $X_1X_4^k=r^{2k}X_4^kX_1+\displaystyle\frac{r^{2k}-s^{2k}}{r^2-s^2}X_4^{k-1}X_2-\frac{s^2\left(r^k-s^k\right)\left(r^{k-1}-s^{k-1}\right)}{\left(r-s\right)\left(r^2-s^2\right)}X_4^{k-2}X_3$ for $k \geq 2$.
        \item $X_1^kX_4=r^{2k}X_4X_1^k+r^{2(k-1)}\displaystyle\frac{1-(r^{-1}s)^{2k}}{1-(r^{-1}s)^2}X_2X_1^{k-1}$.
        \item $X_4^kX_2=s^{-2k}X_2X_4^k+(rs)^{-(k-1)}\displaystyle\frac{1-(rs^{-1})^k}{1-rs^{-1}}X_3X_4^{k-1}$.
        \item $X_2^kX_4=s^{2k}X_4X_2^k-s^{2k}\displaystyle\frac{1-(rs^{-1})^{k}}{1-rs^{-1}}X_3X_2^{k-1}$.
    \end{enumerate}
\end{lemma}
\begin{proof}
    These identities can be easily proved by induction on $k$.
\end{proof}
\begin{remark}\label{rem2} If $r$ and $s$ are $p$-th root of unity then
    $X_1^p, X_2^p,X_3^p$ and $X_4^p$ are central elements in $ U_{r,s}^+(B_2)$.
\end{remark}
\subsection{Polynomial Identity Algebras} In the roots of unity setting, the algebra $ U_{r,s}^+(B_2)$ becomes a finitely generated module over its center, by Remark \ref{rem2}. Hence as a result \cite[Corollary 13.1.13]{mcr}, the algebra $ U_{r,s}^+(B_2)$ becomes a polynomial identity algebra. This sufficient condition on the parameters to be PI algebra is also necessary.
\begin{theorem}\label{m1}
    The algebra $ U_{r,s}^+(B_2)$ is a PI algebra if and only if $r$ and $s$ are roots of unity.
\end{theorem}
\begin{proof}
    The sufficient part is already done. For the necessary part, we use the fact that any subalgebra of a PI algebra is also a PI algebra. Consider the subalgebra of $ U_{r,s}^+(B_2)$ generated by $X_2$ and $X_3$, which is a quantum affine space. It is a well-known result that if $r$ and $s$ are not roots of unity, then the quantum plane $\mathbb{C}\langle X_2,X_3\rangle /\langle X_2X_3-rsX_3X_2\rangle$ is not a PI algebra \cite[Proposition I.14.2]{brg}. Hence, the result follows.
\end{proof}
Kaplansky's Theorem has a striking consequence in the case of a prime affine PI algebra over an algebraically closed field.
\begin{prop}\label{pidimresult}\emph{(\cite[Theorem I.13.5]{brg})}
Let $A$ be a prime affine PI algebra over an algebraically closed field $\mathbb{K}$ and $V$ be a simple $A$-module. Then $V$ is a finite dimensional vector space over $\mathbb{K}$ with $\dime_{\mathbb{K}}(V)\leq \pideg \,(A)$.
\end{prop} 
This result provides an important link between the PI degree of a prime affine PI algebra over an algebraically closed field and its irreducible representations. Moreover, the upper bound PI-deg\,($A$) is attained for such an algebra $A$ (cf. \cite[Lemma III.1.2]{brg}). 
\begin{remark}\label{fdb}
The algebra $ U_{r,s}^+(B_2)$ is classified as a prime affine PI algebra. From the above discussion, it is quite clear that each simple $ U_{r,s}^+(B_2)$-module is finite dimensional and can have dimension at most $\pideg \,\left({ U_{r,s}^+(B_2)}\right)$. Therefore the computation of PI degree for $ U_{r,s}^+(B_2)$ is of substantial importance. Section \ref{secpideg} is devoted to this computation. 
\end{remark}
\begin{remark}
    As stated in our assumptions, we work under the condition that $ r^2 \neq s^2$. Indeed, if $r^2 = s^2$, then the algebra $U_{r,s}^+(B_2)$ fails to be a PI algebra. For a detailed proof, see the Appendix.

\end{remark}
\section{\texorpdfstring{PI degree for $ U_{r,s}^+(B_2) $}{}}\label{secpideg}
In this section, we derive an explicit formula for the PI‑degree of the algebra $U_{r,s}^+(B_2)$ when specialized at roots of unity. We leverage  a pivotal method by De Concini–Procesi for determining the PI‑degree of quantum affine spaces (see \cite[Proposition 7.1]{di}).

Recall that for a multiplicatively skew-symmetric matrix \(\Lambda = (\lambda_{ij})\), the quantum affine space
\[
\mathcal{O}_{\Lambda}(\mathbb{K}^n)
= \mathbb{K}\bigl\langle x_1, \dots, x_n \mid x_i x_j = \lambda_{ij} x_j x_i\ \forall\,i,j \bigr\rangle
\]
is known to be a PI‑algebra when the parameters are roots of unity (see \cite[Proposition I.14.2]{brg}). The following proposition refines the procedure of De Concini–Procesi by employing properties of the associated multiplicatively skew‑symmetric integer matrix \(\Lambda\).

\begin{prop}[{\cite[Lemma~5.7]{ar}}]\label{mainpi}
Let \( q \in \mathbb{K}^* \) be a primitive \( l \)-th root of unity, and let 
$
\lambda_{ij} = q^{h_{ij}},
$
where \( H = (h_{ij}) \) is a skew-symmetric integer matrix of rank \( 2t \).  
Assume that the invariant factors of \( H \) are \( h_1, \ldots, h_t \) (each occurring in pairs).  
Then the PI degree of the corresponding quantum affine space algebra \(\mathcal{O}_{\Lambda}(\mathbb{K}^n)\) is given by
\[
\pideg\,\left(\mathcal{O}_{\Lambda}(\mathbb{K}^n)\right)
= \prod_{i=1}^{t} \frac{l}{\gcd(h_i,\, l)}.
\]
\end{prop}

Assuming that $r$ and $s$ are primitive $m$-th and $n$-th roots of unity, respectively, in the following subsections, we explicitly determine the PI degree of $U_{r,s}^+(B_2)$.
\subsection{\texorpdfstring{PI degree of $U_{r,s}^+(B_2)$ when $rs=1$.}{}} Note in this case the algebra $U_{r,r^{-1}}^+(B_2)$ becomes the uniparameter version $U_r^+(B_2)$ whose PI degree is already known to be $\ord(r^2)$ \cite[Theorem 2.7]{smsb}.

\subsection{\texorpdfstring{PI degree of \( U_{r,s}^+(B_2) \) when \( rs \neq 1 \)}{}}\label{Sub3.2}

From Subsection \ref{Sub2.2}, the algebra \( U_{r,s}^+(B_2) \) admits an iterated skew polynomial ring presentation. Hence, by \cite[Corollary~I.14.I]{brg}, we have
\[
\pideg\,\left(U_{r,s}^+(B_2)\right) 
= \pideg\,\left(\mathcal{O}_{H}(\mathbb{C}^4)\right),
\]
where \( H \) is the \(4 \times 4\) matrix
\[
H := \begin{pmatrix}
  1 & s^2 & r^2 s^2 & r^2 \\
  s^{-2} & 1 & r s & s^2 \\
  r^{-2} s^{-2} & r^{-1} s^{-1} & 1 & r s \\
  r^{-2} & s^{-2} & r^{-1} s^{-1} & 1
\end{pmatrix}.
\]

Let $\Gamma$ be the multiplicative subgroup of $\mathbb{K}^*$ generated by $r$ and $s$. Then $\Gamma$ is a finite cyclic group, say generated by some element $q \in \mathbb{K}^*$. The order of $\Gamma$ is equal to the least common multiple of $m$ and $n$, i.e.,
$$
|\Gamma| = \operatorname{lcm}(m, n) = l.
$$
Choose
$$
s_1 = \frac{n}{\gcd(m, n)} \quad \text{and} \quad s_2 = \frac{m}{\gcd(m, n)}.
$$
Then the subgroups generated by $ q^{s_1} $ and $ q^{s_2} $ are:
$$
\langle q^{s_1} \rangle = \langle r \rangle \quad \text{and} \quad \langle q^{s_2} \rangle = \langle s \rangle.
$$
Therefore, we can write 
$$
r = q^{s_1 k_1} \quad \text{and} \quad s = q^{s_2 k_2}
$$
for some integers $ 1 \leq k_1 < m $ and $ 1 \leq k_2 < n $, such that 
$$
\gcd(k_1, m) = 1 \quad \text{and} \quad \gcd(k_2, n) = 1.
$$
The skew-symmetric matrix associated to the algebra $ U_{r,s}^+(B_2)$ corresponding to the generators $X_1$, $X_2$, $X_3$ and $X_4$ is:
$$
A=\begin{pmatrix}
0 & 2s_2k_2 & 2(s_1k_1 + s_2k_2) & 2s_1k_1 \\
-2s_2k_2 & 0 & s_1k_1 + s_2k_2 & 2s_2k_2 \\
-2(s_1k_1 + s_2k_2) & -(s_1k_1 + s_2k_2) & 0 & s_1k_1 + s_2k_2 \\
-2s_1k_1 & -2s_2k_2 & -(s_1k_1 + s_2k_2) & 0
\end{pmatrix}.
$$
Since $rs\neq 1$, the matrix $A$ is non-singular and therefore has two invariant factors, denoted as $h_1$ and $h_2$, where $h_1$ is the first determinantal divisor and $h_2$ can be determined by the fact that $h_1h_2=\sqrt{\operatorname{det (A)}}$. It is straightforward to compute that: 
\begin{equation}\label{eq3.1} h_1=\gcd(s_1k_1+s_2k_2, s_1k_1-s_2k_2),\end{equation}  \begin{equation}\label{eq3.2}
    h_1h_2= 2(s_1k_1+s_2k_2)(s_1k_1-s_2k_2).\end{equation}
Then from Proposition \ref{mainpi}, the PI degree of the algebra $ U_{r,s}^+(B_2)$ is given by:
\begin{equation}\label{eq3.3}
\operatorname{PI-deg}\,  \left(U_{r,s}^+(B_2)\right)= \frac{l^2}{\gcd(h_1,l)\gcd(h_2,l)}.
\end{equation}
We will calculate the PI degree of $U_{r,s}^+(B_2)$ explicitly based on two cases, whether $l$ is odd or even.

\noindent\textbf{Case 1: $l$ is odd.} We claim that $\gcd(h_1, l) = 1$. Suppose, for contradiction, that $\gcd(h_1, l) = d > 1$. Since $l$ is odd, $d$ must also be odd. Then $d \mid s_1k_1$ and $d \mid s_2k_2$. Since $r = q^{s_1k_1}$ and $s = q^{s_2k_2}$, it follows that $r^{l/d} = 1$ and $s^{l/d} = 1$. This implies that $\operatorname{ord}(r)$ and $\operatorname{ord}(s)$ both divide $l/d$, contradicting the fact that $l = \operatorname{lcm}(m, n)$. Hence, $d = 1$. 

\noindent Since $\gcd(h_1, l) = 1$, it follows that $\gcd(h_2, l) = \gcd(h_1 h_2, l)$. Since $l$ is odd, using the Equation (\ref{eq3.2}) in Equation (\ref{eq3.3}), we obtain
$$
\operatorname{PI-deg}\, \left(U_{r,s}^+(B_2)\right) = \frac{l^2}{\gcd\big((s_1k_1 + s_2k_2)(s_1k_1 - s_2k_2), l\big)}.
$$
Since $\gcd(h_1, l) = 1$, we deduce that:
$$
\gcd\big((s_1k_1 + s_2k_2)(s_1k_1 - s_2k_2), l\big) = \gcd\left( \frac{(s_1k_1 + s_2k_2)}{h_1} \times \frac{(s_1k_1 - s_2k_2)}{h_1}, l\right).
$$
It is a well-known result in number theory that for any integers $a, b, c \in \mathbb{Z}$, if $\gcd(a, b) = 1$, then
$\gcd(ab, c) = \gcd(a, c)\gcd(b, c).$ Therefore, by using Equation (\ref{eq3.1}) and then again by using $\gcd(h_1, l) = 1$ we have
\begin{align*}
\operatorname{PI-deg}\, \left(U_{r,s}^+(B_2)\right) &= \frac{l^2}{\gcd\big((s_1k_1 + s_2k_2), l\big)\gcd\big((s_1k_1 - s_2k_2), l\big)}\\[1em]
&=\ord(r^2s^2)\ord(r^2s^{-2}).
\end{align*}

\noindent\textbf{Case 2: $l$ is even.} Since $l$ is even, at least one of $\ord(r)$ or $\ord(s)$ must be even.

\noindent \textbf{Claim 1:} If $e_2(m) \neq e_2(n)$, then $h_1$ is odd.

\noindent\textit{Proof of Claim 1:}  Without loss of generality, assume $e_2(m) < e_2(n)$. Let $m = 2^{e_2(m)} m'$ and $n = 2^{e_2(n)} n'$, where $m'$ and $n'$ are odd. Then $l = 2^{e_2(n)} p'$ where $p' = \operatorname{lcm}(m',n')$ is odd. So $s_1 = l/m$ is even and $s_2 = l/n$ is odd. Since $\gcd(k_2, n) = 1$ and $n$ is even, $k_2$ must be odd. Therefore, $s_1k_1 \text{ is even}, s_2k_2 \text{ is odd}$ implies that $s_1k_1 \pm s_2k_2 \text{ are odd}.$ Hence, from Equation (\ref{eq3.1}), $h_1 $ is odd.

\noindent \textbf{Claim 2:} If $e_2(m) = e_2(n)$, then $e_2(h_1)=1.$ Consequently, we have $\gcd\left(h_1, l\right) = 2$, and hence $\gcd(h_2,l)=\gcd(h_1h_2/2,l)$.

\noindent\textit{Proof of Claim 2:} Following a similar argument as in the proof of Claim 1, we conclude that both $s_1k_1$ and $s_2k_2$ are odd implies that $s_1k_1 \pm s_2k_2 \text{ are even}.$ Hence, we have $h_1 = 2\gcd(s_1k_1, s_2k_2)$, which proves that $e_2(h_1)=1$. Since both $h_1$ and $l$ are even, in fact $h_1$ is a multiple of 2 and some odd number. Using a similar argument as in Case 1, we can easily prove that $\gcd(h_1,l)=2$. Since $e_2(h_1)=1$, therefore $\gcd(h_1/2,l)=1$ and using this we have $\gcd(h_2,l)=\gcd(h_1h_2/2,l)$.

\noindent\textbf{Subcase 2.1: $e_2(m) \neq e_2(n)$}.

\noindent From Claim 1, $h_1$ is odd, using a similar argument as in Case 1, it is easy to prove that $\gcd(h_1,l/2)=1$, and again, a similar argument as in Case 1, we obtain
$$
\operatorname{PI-deg}\, \left(U_{r,s}^+(B_2)\right) = 2  \ord(r^2s^2)  \ord(r^2s^{-2}).
$$
\noindent\textbf{Subcase 2.2: $e_2(m)= e_2(n)=1$}. 

\noindent Using  Claim 2, Equation (\ref{eq3.2}) and the fact that $\gcd(h_1/2,l)=1$, we have $$\gcd(h_2,l)=\gcd\left( \frac{(s_1k_1 + s_2k_2)}{h_1/2} \times \frac{(s_1k_1 - s_2k_2)}{h_1/2}, l\right).$$
Since $l/2$ is odd, using Equation (\ref{eq3.1}) along with $\gcd(h_1,l/2)=1$, we obtain
$$\gcd(h_2,l)= 2\gcd\left(s_1k_1 + s_2k_2,\frac{l}{2} \right)\gcd\left(s_1k_1 - s_2k_2,\frac{l}{2} \right).$$
Therefore, from Equation (\ref{eq3.3}), we get
$$\operatorname{PI-deg}\, \left(U_{r,s}^+(B_2)\right)= \ord(r^2s^2)  \ord(r^2s^{-2}).$$
\noindent\textbf{Subcase 2.3: $e_2(m)= e_2(n)=k\geq 2$}.
\noindent Using  Claim 2, Equation (\ref{eq3.2}), and $\gcd(h_1/2,l)=1$, we have $$\gcd(h_2,l)=4\gcd\left( \frac{(s_1k_1 + s_2k_2)}{h_1} \times \frac{(s_1k_1 - s_2k_2)}{h_1}, \frac{l}{4}\right).$$
Using Equation (\ref{eq3.1}), we obtain
\begin{align*}
  \gcd(h_2,l)&= 4\gcd\left(\frac{s_1k_1 + s_2k_2}{h_1},\frac{l}{4} \right)\gcd\left(\frac{s_1k_1 - s_2k_2}{h_1},\frac{l}{4} \right)\\
  &=\gcd\left(\frac{s_1k_1 + s_2k_2}{h_1/2},\frac{l}{2} \right)\gcd\left(\frac{s_1k_1 - s_2k_2}{h_1/2},\frac{l}{2} \right).
\end{align*}
Using Claim 2, we have $\gcd(h_1/2,l/2)=1$. Therefore from equitation (\ref{eq3.3}), we get
$$\operatorname{PI-deg}\, \left(U_{r,s}^+(B_2)\right)= 2\ord(r^2s^2)  \ord(r^2s^{-2}).$$
Hence, we have the following theorem:
\begin{theorem} \label{m2}
let $m=\ord(r)$ and $n=\ord(s)$. If both $m$ and $n$ are odd then
    \[\pideg \,\left(U_{r,s}^+(B_2)\right)=\ord(r^2s^2)\ord(r^2s^{-2}).\] If either one of $m$ or $n$ is even then 
    \begin{align*}
        \pideg \,\left(U_{r,s}^+(B_2)\right)=\begin{cases}
           2\ord(r^2s^2)\ord(r^2s^{-2}), & e_2(m)\neq e_2(n)\\
            \ord(r^2s^2)\ord(r^2s^{-2}), &e_2(m)=e_2(n)=1\\
             2\ord(r^2s^2)\ord(r^2s^{-2}), &e_2(m)=e_2(n)=k \geq 2.
        \end{cases}
    \end{align*}
    \end{theorem}
    It is worth noting when
    \begin{itemize}
        \item $l$ is odd both $\ord(rs)$ and $\ord(rs^{-1})$ are odd and hence PI degree is equal to $\ord(rs)\ord(rs^{-1})$.
        \item  $l$ is even with $e_2(m)\neq e_2(n)$ then $\ord(rs)$ and $\ord(rs^{-1})$ are both even and hence PI degree is $\displaystyle\frac{\ord(rs)\ord(rs^{-1})}{2}$.
        \item $l$ is even with $e_2(m)= e_2(n)=1$ then $\ord(rs)$ and $\ord(rs^{-1})$ are both odd and hence PI degree is equal to $\ord(rs)\ord(rs^{-1})$.
        \item $l$ is even with $e_2(m)= e_2(n)=k \geq 2$ then it is easy to observe one of $s_1k_1\pm s_2k_2$ is an odd multiple of $2$ and the other is a multiple of $4$. This ensures at least one of $\ord(rs)$ and $\ord(rs^{-1})$ is even. The PI degree is $\ord(rs)\ord(rs^{-1})$ if exactly one is even and $\displaystyle\frac{\ord(rs)\ord(rs^{-1})}{2}$ if both are even.
    \end{itemize}
In the following sections, our primary objective is to construct and classify simple $U_{r,s}^+(B_2)$-modules.
\section{\texorpdfstring{An \enquote{Approximation} subalgebra $\mathcal{B}$ of $U_{r,s}^+(B_2)$.}{} }
It is more straightforward to begin by constructing the representations of a subalgebra of $U_{r,s}^+(B_2)$. This subalgebra has the advantage that, while its commutation relations are considerably simpler than those of $U_{r,s}^+(B_2)$, its representations can be directly utilized to construct representations of $U_{r,s}^+(B_2)$. In this section, we introduce this subalgebra, which we will denote by $\mathcal{B}$, and examine its structure in detail. We define \begin{equation}\label{cen1}
    \tilde{W}:=X_2+(r^2-s^2)X_4X_1.
\end{equation} Consider the $\C$ subalgebra $\B$ of $U_{r,s}^+(B_2)$ generated by the elements $X_1$, $ X_2,$  $X_3$ and $\tilde{W}$. Then the relations between the generators are the following
\begin{align*}
    \tilde{W}X_1=r^{-2}X_1\tilde{W}, \ X_2X_1&=s^{-2}X_1X_2, \ X_3X_1=(rs)^{-2}X_1X_3, \  X_3X_2=(rs)^{-1}X_2X_3\\
    \tilde{W}X_3&=rsX_3\tilde{W}, \ \tilde{W}X_2=X_2\tilde{W}+s^2(r^2-s^2)X_3X_1.
\end{align*}
Note that the algebra $\B$ can be presented as an iterated skew polynomial ring 
\[\B=\C[X_1][X_2,\tau'_2][X_3,\tau'_3][X_4,\tau'_4;\delta'_4]\]
where $\tau'_2$ is an automorphism of $\C[X_1]$ defined by $\tau'_2(X_1)=s^{-2}X_1$, $\tau'_3$ is an automorphism of $\C[X_1][X_2,\tau'_2]$ defined by $\tau'_3(X_1)=(rs)^{-2}X_1$, $\tau'_3(X_2)=(rs)^{-1}X_2$ and finally $\tau'_4$ is an automorphism of $\C[X_1][X_2,\tau'_2][X_3,\tau'_3]$ defined by $\tau'_4(X_1)=r^{-2}X_1$, $\tau'_4(X_2)=X_2$ and $\tau'_4(X_3)=rsX_3$. Also $\delta'_4$ is a $\tau'_4$-derivation of $\C[X_1][X_2,\tau'_2][X_3,\tau'_3]$ where $\delta'_4(X_1)=0, \delta'_4(X_2)=s^2(r^2-s^2)X_3X_1$ and $\delta'_4(X_3)=0$.

\begin{thm}
    The following identities hold in $\mathcal{B}$.
    \begin{enumerate}[label= \((\arabic*)\)]
        \item $X_2^a\tilde{W}=\tilde{W}X_2^a-s^2(r^2-s^2)\displaystyle\frac{1-(rs^{-1})^{a}}{1-rs^{-1}}X_3X_1X_2^{a-1}$.
        \item $\tilde{W}^aX_2=X_2\tilde{W}^a+s^2(r^2-s^2)\displaystyle\frac{1-(r^{-1}s)^{a}}{1-r^{-1}s}X_3X_1\tilde{W}^{a-1}$.
    \end{enumerate}
\end{thm}
\begin{proof}
    These identities can be easily proved by induction on $a$.
\end{proof}
\begin{cor}\label{cen2}
  $X_1^l, X_2^l, X_3^l$ and $\tilde{W}^l$ are central elements in $\B$ where $l=\operatorname{lcm}(m,n)$.  
\end{cor}
\begin{thm}
   The algebra $\mathcal{B}$ is a prime PI algebra.
\end{thm}
\begin{proof}
    Let $\mathscr{C}= \mathbb{K}[X_1^l,X_2^l, X_3^l, \tilde{W}^l]$ be the polynomial algebra. The Corollary \ref{cen2} clearly shows that $\mathscr{C}$ is a central subalgebra of $\mathcal{B}$. Observe that $\mathcal{B}$ is a finitely generated module over $\mathscr{C}$. Thus, the conclusion is derived from Proposition \cite[Corollary 13.1.13]{mcr}.
\end{proof}
Using the expression of $\tilde{W}$ given in (\ref{cen1}), it can be easily verified that
\begin{equation}\label{sameloc}
    U_{r,s}^+(B_2)[X_1^{-1}]=\mathcal{B}[X_1^{-1}]
\end{equation} where $U_{r,s}^+(B_2)[X_1^{-1}]$ and $\mathcal{B}[X_1^{-1}]$ denote the localizations of $U_{r,s}^+(B_2)$ and $\mathcal{B}$, respectively, at the multiplicative set generated by powers of the element $X_1$. It then follows from \cite[Corollary I.13.3]{brg}, a result on PI degree parity, that 
\[\pideg \,\left(U_{r,s}^+(B_2)\right)=\pideg \,\left(U_{r,s}^+(B_2)[X_1^{-1}]\right)=\pideg \,\left(\mathcal{B}[X_1^{-1}]\right)=\pideg \,\left(\mathcal{B}\right).\]
From \cite[Theorem I.13.5]{brg}, we conclude if $M$ is a $\mathcal{B}$-simple module, then $\dime_{\C}M \leq \pideg\,(\mathcal{B})$.\\
Now consider the subalgebra $\mathcal{C}$ of $\B$ generated by the elements $X_2, \tilde{W}$ and $X_3X_1$. Then the relations between the generators are the following
\[X_2(X_3X_1)=rs^{-1}(X_3X_1)X_2,\]
\[  \tilde{W}(X_3X_1)=r^{-1}s(X_3X_1)\tilde{W},
\quad \tilde{W}X_2=X_2\tilde{W}+s^2(r^2-s^2)X_3X_1.\]
Now we use the concept of the generalized Weyl algebra introduced by V. Bavula, a concept crucial in identifying some crucial normal elements of the algebra $\B$. We begin by defining the generalized Weyl algebra.
\begin{dfn}\cite{ba3}
    Let $D$ be a ring, $\sigma$ be an automorphism of $D$
and $a$ is an element of the center of $D$. The generalized Weyl algebra $A:=D[X,Y;\sigma,a]$ is a ring generated by $D, X$ and $Y$ subject to the defining relations:
\[X\alpha=\sigma(\alpha)X,~ Y\alpha=\sigma^{-1}(\alpha)Y~\text{for all $\alpha \in D$},~ YX=a,~XY=\sigma(a).\]
\end{dfn}
In this article, we will show that many subalgebras of $U_{r,s}^+(B_2)$ are GWAs.
\begin{dfn}\cite{ba3}
    Let $D$ be a ring and $\sigma$ be its automorphism. Suppose that elements $b$ and $\rho$ belong to the centre of the ring $D$, $\rho$ is invertible and $\sigma(\rho)=\rho$.
Then $E=D[X,Y;\sigma,\rho,b]$ is a ring generated by $D, X$ and $Y$ subject to the defining relations:
\[X\alpha=\sigma(\alpha)X,~ Y\alpha=\sigma^{-1}(\alpha)Y~\text{for all $\alpha \in D$},~XY-\rho YX=b.\]
\end{dfn}
The next proposition shows that the rings $E$ are GWAs and, under a (mild) condition, they have a \enquote{canonical} normal element.
\begin{prop}\label{GWA}\cite[Lemma 1.3, Corollary 1.4, Corollary 1.6]{ba3}
    Let $E=D[X,Y;\sigma,\rho,b]$. Then 
    \begin{enumerate}[label= \((\arabic*)\)]
        \item The following statements are equivalent
        \begin{enumerate}[label=\((\alph*)\)]
            \item ${C}=\rho(YX+\alpha)=XY+\sigma(\alpha)$  is a normal element in $E$ for some central element $\alpha \in D$.
            \item $\rho\alpha-\sigma(\alpha)=b$ for some central element $\alpha \in D$.
        \end{enumerate}
        \item  If one of the equivalent conditions of statement 1 holds then the ring $E=D[C][X, Y,\sigma,a=\rho^{-1}C-\alpha]$  is a GWA where $\sigma(C)=\rho C$. 
    \end{enumerate}
    Moreover, if $\rho=1$ then the element $C$ is central in $E$.
\end{prop}
\textbf{The algebra $\mathcal{C}$ is a GWA.} Note $\mathcal{C}=\C[X_3X_1][\tilde{W}, X_2;\sigma, \rho:=1, b:=s^2(r^2-s^2)X_3X_1]$ where $\sigma(X_3X_1)=r^{-1}sX_3X_1$. The element $\displaystyle\frac{s^2(r^2-s^2)X_3X_1}{1-r^{-1}s} \in \mathbb{K}[X_3X_1]$ is a solution of equation $\alpha-\sigma(\alpha)=s^2(r^2-s^2)X_3X_1$. Hence, by Proposition \ref{GWA}, the element \begin{equation}\label{cen3}
     \tilde{X}:=\tilde{W}X_2+\displaystyle\frac{s^2(r^2-s^2)r^{-1}s}{1-r^{-1}s}X_3X_1=\tilde{W}X_2-\displaystyle\frac{s^2(r^2-s^2)}{1-rs^{-1}}X_3X_1
 \end{equation} is a central element of $\mathcal{C}$ i.e., $\tilde{X}$ commutes with $X_3X_1, \tilde{W}$ and $X_2$. Also by Proposition \ref{GWA}, we conclude that $\mathcal{C}$ is a GWA. Observe that $\tilde{X}X_3=r^2s^2X_3\tilde{X}$ and $\tilde{X}X_1=r^{-2}s^{-2}X_1\tilde{X}$. So $\tilde{X}$ is also a normal element in $\mathcal{B}$. Now we have the following well known result.
 \begin{thm}
     Suppose that $A$ is an algebra, $x$ is a normal element of $A$ and $M$ is a simple $A$-module. Then either $Mx=0$ (if $M$ is $x$-torsion) or the map $x_{M}:M\rightarrow M, m\mapsto mx$ is a bijection (if $M$ is $x$-torsion-free).
 \end{thm}
 This lemma implies that the action of a normal element on a simple module is either zero or invertible. Note that $\tilde{X}$ is a normal element in $\mathcal{B}$ and hence the multiplicative set $S=\{\tilde{X}^k:k\geq 0\}$ forms an Ore set in $\mathcal{B}$, which allows us to consider notions of $\tilde{X}$-torsion and $\tilde{X}$-torsionfree modules. If $N$ is a simple $\mathcal{B}$-module, then either $\tors_{\tilde{X}}(N)=\{0\}$ or $\tors_{\tilde{X}}(N)=N$. This splits the classification of simple $\mathcal{B}$-modules into the following two types:
\begin{enumerate}
    \item[] Type I: Simple $\tilde{X}$-torsionfree $\mathcal{B}$-modules.
    \item[] Type II: Simple $\tilde{X}$-torsion $\mathcal{B}$-modules.
\end{enumerate}
\begin{remark}\label{rem1}
  Let $N$ be a simple $\tilde{X}$-torsion $\mathcal{B}$-module. Then $N\tilde{X}=0$ and $N$ becomes a simple module over $\tilde{\mathcal{B}}:=\mathcal{B}/\langle \tilde{X}\rangle$ which is naturally isomorphic to the factor algebra $\mathcal{O}_P(\mathbb{C}^4)/J$, where $\mathcal{O}_P(\mathbb{C}^4)$ is the quantum affine space  generated by $x_1, x_2, x_3$, and $\tilde{w}$, subject to the relations
 \begin{align*}
     x_1x_2=s^2x_2x_1, \ x_1x_3&=r^2s^2x_3x_1, \ x_1\tilde{w}=r^2\tilde{w}x_1,\\
      x_2x_3=rsx_3x_2, \ x_2\tilde{w}&=rs^{-1}\tilde{w}x_2, \ x_3\tilde{w}=(rs)^{-1}\tilde{w}x_3.
 \end{align*} The matrix $P$ of commutation parameters corresponding to these generators is
 \[P:=\begin{pmatrix}
  1& s^2&r^2s^2&r^2\\
  s^{-2}&1&rs&rs^{-1}\\
  r^{-2}s^{-2}&r^{-1}s^{-1}&1&r^{-1}s^{-1}\\
  r^{-2}&r^{-1}s&rs&1
\end{pmatrix}.\] The ideal $J$ is the two sided ideal generated by the element $$\tilde{w}x_2-\displaystyle\frac{s^2(r^2-s^2)}{1-rs^{-1}}x_3x_1.$$
\end{remark}
\begin{thm}
The algebra  $\tilde{\mathcal{B}}$ is a PI algebra.  
\end{thm}
\begin{proof}
    Easily follows from the fact that $\mathcal{B}$ is a polynomial identity algebra and $\tilde{\mathcal{B}}$ is a factor algebra of $\mathcal{B}$.
\end{proof}
\begin{thm}
    \begin{align*}
        \pideg\,\left(\tilde{\mathcal{B}}\right)&\leq\begin{cases}
        \displaystyle\frac{l}{2},  & \text{if m and n are both even with $e_2(m)=e_2(n)$}\\
        l, & otherwise.
    \end{cases}
    \end{align*}
\end{thm}
\begin{proof}
    Since $\tilde{\mathcal{B}}$ is a quotient of the quantum affine space $\mathcal{O}_P(\mathbb{C}^4)$, it suffices to compute $\pideg(\mathcal{O}_P(\mathbb{C}^4))$.
    
   \noindent The skew-symmetric matrix $P'$ associated with the matrix $P$ is the following
    $$
P'=\begin{pmatrix}
0 & 2s_2k_2 & 2(s_1k_1 + s_2k_2) & 2s_1k_1 \\
-2s_2k_2 & 0 & s_1k_1 + s_2k_2 & s_1k_1-s_2k_2 \\
-2(s_1k_1 + s_2k_2) & -(s_1k_1 + s_2k_2) & 0 & -s_1k_1 - s_2k_2 \\
-2s_1k_1 & -s_1k_1+s_2k_2 & s_1k_1 + s_2k_2 & 0
\end{pmatrix}.
$$ We can easily observe that det$(P')=0$. Hence $\ran(P')=2$, so $h_1$ is the only invariant factor of $P'$. This $h_1$ is the first determinantal divisor, given by $$h_1=\gcd{(s_1k_1+s_2k_2,s_1k_1-s_2k_2)}.$$
Then, following the same method as in the calculation of the PI degree of $U_{r,s}^+(B_2)$, we obtain 
\[\pideg\, \left({\mathcal{O}_P(\mathbb{C}^4)}\right)=\begin{cases}
        \displaystyle\frac{l}{2},  & \text{if m and n are both even with $e_2(m)=e_2(n)$}\\
        l, & otherwise.
    \end{cases}\]
\end{proof}
\begin{remark}
    It is easy to prove that $\pideg\,
\left(\mathcal{O}_P(\mathbb{C}^4)\right)=\ell$ where $\ell=\operatorname{lcm}(\ord(rs),\ord(rs^{-1}))$.
\end{remark}
Also observe $X_1$ is normal in $\mathcal{B}$ and hence any simple $\mathcal{B}$-module is either $X_1$-torsion-free or $X_1$-torsion.
The following theorem justifies the adjective ``approximation'' for the subalgebra $\mathcal{B}$.
 \begin{thm}\label{itd}
There is a one-to-one correspondence between the simple $X_1$-torsionfree $\mathcal{B}$-modules and the simple $X_1$-torsionfree $U_{r,s}^{+}(B_2)$-modules.
\end{thm}
\begin{proof}
   From Theorem \ref{m1} along with Proposition \ref{pidimresult}, it is quite clear that each simple $U_{r,s}^{+}(B_2)$-module is finite dimensional and can have dimension at most $\pideg\,\left(U_{r,s}^{+}(B_2)\right)$. The same is true for algebra $\mathcal{B}$ also.

 \noindent Let $N$ be a simple $X_1$-torsionfree $\mathcal{B}$-module. Since $X_1$ is normal in $\mathcal{B}$, it acts invertibly on $N$. Using this fact and the expression of $\tilde{W}$  from Equation (\ref{cen1}), we can define the action of $X_4$ on $N$ explicitly 
\[vX_4=\displaystyle\frac{1}{r^2-s^2}\left(v\tilde{W}-vX_2\right)X_1^{-1}, \quad \text{for each $v\in N$}.\]
Thus $N$ becomes an $U_{r,s}^{+}(B_2)$-module which is simple as well. So any simple $X_1$-torsionfree $\mathcal{B}$-module gives a simple $X_1$-torsionfree $U_{r,s}^{+}(B_2)$-module.

 \noindent On the other hand suppose $M$ be a simple $X_1$-torsionfree $U_{r,s}^{+}(B_2)$-module. Clearly, $M$ is also a finite dimensional $X_1$-torsionfree $\mathcal{B}$-module. Suppose $M'$ is a nonzero simple $\mathcal{B}$-submodule of $M$. Since $X_1$ acts invertibly on $M'$, we can again define the action of $X_4$ on $M'$. Then $M'$ becomes a $U_{r,s}^{+}(B_2)$-module with $M'\subseteq M$. By simplicity of $M$, we get $M=M'$. Hence $M$ itself is a simple $X_1$-torsionfree $\mathcal{B}$-module.  
\end{proof}
 \section{\texorpdfstring{Construction of $X_1$-torsionfree simple $\mathcal{B}$-modules}{}}\label{Sec5}

In this section, we focus on constructing simple modules over the subalgebra $ \mathcal{B} \subseteq U^+_{r,s}(B_2) $, which are torsionfree with respect to the action of the normal element $ X_1 $. These $ \mathcal{B} $-modules serve as the foundational building blocks for understanding representations of the larger algebra. We classify such modules into three distinct types—$ M_1(\lambda) $, $ M_2(\mu) $, and $ M_3(\epsilon) $—based on their torsion behavior and eigenvalue conditions. Each type is explicitly constructed and shown to be simple via careful analysis of the module actions and internal parameters.

 \subsection{\texorpdfstring{ Construction of $X_1$-torsionfree simple $\mathcal{B}$-modules of Type I}{}} In this subsection, we concentrate on the construction of Type I simple $X_1$-torsionfree $\mathcal{B}$-modules.
 
\subsubsection{\texorpdfstring {\textbf{\upshape Simple $\mathcal{B}$-modules of type $\mathcal{M}_1(\lambda)$}}{}}
Let $ \lambda = (\lambda_1, \lambda_2, \lambda_3, \lambda_4) \in (\mathbb{C}^*)^3 \times \mathbb{C} $ be a parameter satisfying 
$$
    \lambda_4 \neq \frac{s^2 (r^2 - s^2)}{1 - rs^{-1}} \lambda_3.
$$  
Define $ \mathcal{M}_1(\lambda) $ as a vector space with basis  
$$
    \{ e(a, b) \mid 0 \leq a \leq \ell_1 - 1, \ 0 \leq b \leq m_1 - 1 \},
$$  
where $ \ell_1, m_1 \in \mathbb{N} $ are given by  
\begin{equation}\label{Eq5.1}
    \ell_1 := 
    \begin{cases}
        \ord(r^{2}s^2), & \text{if } \ord(s^2) \mid \ord(r^{2}s^2) \ord(rs^{-1}), \\[5pt]
        2\ord(r^{2}s^2), & \text{otherwise},
    \end{cases} \quad \text{and}\quad  m_1 := \ord(rs^{-1}).
\end{equation}
We define the $ \mathcal{B}$-module structure on $ \mathcal{M}_1(\lambda) $ by
\begin{align}\label{Eq5.2}
    e(a,b)X_1 &= s^{-2b} \lambda_1 e(a \oplus 1, b),\\
    \label{Eq5.3}
    e(a,b)X_2 &= 
    \begin{cases}
        \lambda_2 e(a,b+1), & b \neq m_1 - 1, \\[5pt]
        \displaystyle s^{2 a m_1} \lambda_2 e(a,0), & b = m_1 - 1,
    \end{cases}\\ \label{Eq5.4}
e(a,b)X_3 &= (rs)^{2a+b} \lambda_1^{-1} \lambda_3 e(a \oplus (-1), b),\\
e(a,b)\tilde{W}&=\begin{cases}
    \lambda_2^{-1}(rs)^{2a}\left(\lambda_4-s^2(r^2-s^2)\displaystyle\frac{1-(rs^{-1})^b}{1-rs^{-1}}\lambda_3\right)e(a,b-1), & b \neq 0,\\
    \displaystyle\lambda_2^{-1}\lambda_4 r^{2a}s^{-2a(m_1-1)}e(a,m_1-1), & b =0,
\end{cases}\notag
\end{align} where $\oplus$ denotes the addition modulo $\ell_1$.
Indeed, these actions define a simple $\mathcal{B}$-module.\\
For $b \neq 0,m_1-1$, we have 
\begin{align*}
    &e(a,b)\tilde{W}X_2=\lambda_2^{-1}(rs)^{2a}\left(\lambda_4-s^2(r^2-s^2)\displaystyle\frac{1-(rs^{-1})^b}{1-rs^{-1}}\lambda_3\right)e(a,b-1)X_2\\
    &=(rs)^{2a}\left(\lambda_4-s^2(r^2-s^2)\displaystyle\frac{1-(rs^{-1})^b}{1-rs^{-1}}\lambda_3\right)e(a,b).
\end{align*}
On the other hand,
\begin{align*}
    &e(a,b)\left(X_2\tilde{W}+s^2(r^2-s^2)X_3X_1\right)\\
    &=\lambda_2e(a,b+1)\tilde{W}+s^2(r^2-s^2)(rs)^{2a+b}\lambda_1^{-1}\lambda_3e(a\oplus(-1),b)X_1\\
    &= (rs)^{2a}\left(\lambda_4-s^2(r^2-s^2)\displaystyle\frac{1-(rs^{-1})^{b+1}}{1-rs^{-1}}\lambda_3\right)e(a,b)+s^2(r^2-s^2)(rs)^{2a}(rs^{-1})^b\lambda_3e(a,b)\\
     &= (rs)^{2a}\left[\lambda_4-s^2(r^2-s^2)\lambda_3\left(\displaystyle\frac{1-(rs^{-1})^{b+1}}{1-rs^{-1}}-(rs^{-1})^b\right)\right]e(a,b)\end{align*}
    \begin{align*}
    =& (rs)^{2a}\left(\lambda_4-s^2(r^2-s^2)\displaystyle\frac{1-(rs^{-1})^b}{1-rs^{-1}}\lambda_3\right)e(a,b).
\end{align*}
Next, for $b=0$, we have 
\begin{align*}
    e(a,0)\tilde{W}X_2&=\lambda_2^{-1}\lambda_4r^{2a}s^{-2a(m_1-1)}e(a,m_1-1)X_2=\lambda_4(rs)^{2a}e(a,0).
\end{align*} and 
\begin{align*}
    &e(a,0)\left(X_2\tilde{W}+s^2(r^2-s^2)X_3X_1\right)\\
    &=\lambda_2e(a,1)\tilde{W}+s^2(r^2-s^2)(rs)^{2a}\lambda_1^{-1}\lambda_3e(a\oplus(-1),0)X_1\\
    &=(rs)^{2a}\left(\lambda_4-s^2(r^2-s^2)\lambda_3\right)e(a,0)+s^2(r^2-s^2)(rs)^{2a}e(a,0)\\
    &=\lambda_4(rs)^{2a}e(a,0).
\end{align*}
Finally for $b=m_1-1$, we have 
\begin{align*}
    e(a,m_1-1)\tilde{W}X_2&=\lambda_2^{-1}(rs)^{2a}\left(\lambda_4-s^2(r^2-s^2)\displaystyle\frac{1-(rs^{-1})^{m_1-1}}{1-rs^{-1}}\lambda_3\right)e(a,m_1-2)X_2\\
    &=(rs)^{2a}\left(\lambda_4+s^2(r^2-s^2)r^{-1}s\lambda_3\right)e(a,m_1-1)
\end{align*}
Also,
\begin{align*}
    &e(a,m_1-1)\left(X_2\tilde{W}+s^2(r^2-s^2)X_3X_1\right)\\
    &=s^{2am_1}\lambda_2e(a,0)\tilde{W}+s^2(r^2-s^2)(rs)^{2a+m_1-1}\lambda_1^{-1}\lambda_3e(a \oplus (-1),m_1-1)X_1\\
    &=(rs)^{2a}\lambda_4e(a,m_1-1)+s^2(r^2-s^2)(rs)^{2a+m_1-1}\lambda_1^{-1}\lambda_3s^{-2(m_1-1)}\lambda_1e(a,m_1-1)\\
    &=(rs)^{2a}\left(\lambda_4+s^2(r^2-s^2)r^{-1}s\lambda_3\right)e(a,m_1-1).
\end{align*}
The other relations can easily be verified. Then we have the following theorem
\begin{thm}\label{dim1}
    $\mathcal{M}_1(\lambda)$ is a simple $\mathcal{B}$-module of dimension $\ell_1 m_1$.
\end{thm}
\begin{proof}
    Let $N_1$ be a nonzero submodule of $\mathcal{M}_1(\lambda)$. We claim that $N_1$ contains at least one basis vector. To see this, we define the length of a vector $v \in \mathcal{M}_1(\lambda)$ with respect to the basis $\{e(a,b): 0 \leq a \leq \ell_1-1, 0 \leq b \leq m_1-1\}$ as follows: if $v=\sum_{i=1}^{n}\lambda_{i}e(a^{(i)},b^{(i)})$ where $\lambda_i \in \mathbb{C}^*$, and the basis elements $e(a^{(i)},b^{(i)})$ appearing in the sum are pairwise distinct then we define the length of $v$ to be $n$.
    
    \noindent Let $w=\sum_{\text{finite}}\lambda_{ab}e(a,b)$ be an element in $N_1$ where $\lambda_{ab} \in \mathbb{C}^*$. If $w$ has length $1$, then we are done. If not then there exists distinct pairs $(u,v)$ and $(x,y)$ with $0 \leq u,x \leq \ell_1-1$ and $0 \leq v,y \leq m_1-1$ such that $\lambda_{uv},\lambda_{xy} \neq 0$.
    Note that 
    \begin{align*}
        e(a,b)X_3X_1&=(rs)^{2a}(rs^{-1})^b \lambda_3 e(a,b)\\
        e(a,b)\tilde{W}X_2&=(rs)^{2a}\left(\lambda_4-s^2(r^2-s^2)\displaystyle\frac{1-(rs^{-1})^b}{1-rs^{-1}}\lambda_3\right)e(a,b)\\
        e(a,b)X_2^{m_1}&=s^{2am_1}\lambda_2^{m_1}e(a,b).
    \end{align*}
    Define \[\Lambda_{ab}:=(rs)^{2a}(rs^{-1})^b \lambda_3,\, \Delta_{ab}:=(rs)^{2a}\left(\lambda_4-s^2(r^2-s^2)\displaystyle\frac{1-(rs^{-1})^b}{1-rs^{-1}}\lambda_3\right), \,\Psi_{ab}:=s^{2am_1}\lambda_2^{m_1}.\]  
     Then we have the following cases:

\noindent\textbf{Case I:} $\Lambda_{uv} \neq \Lambda_{xy}$. Then $e(u,v)$ and $e(x,y)$ are eigenvectors of the operator $X_3X_1$ corresponding to distinct eigenvalues. Then $wX_3X_1-\Lambda_{uv}w$ is a nonzero element of $N_1$ having length less than $w$.
    
    \noindent\textbf{Case II:} $\Lambda_{uv} = \Lambda_{xy}$ and $\Delta_{uv}\neq\Delta_{xy}$. Then $e(u,v)$ and $e(x,y)$ are eigenvectors of the operator $\tilde{W}X_2$ corresponding to distinct eigenvalues. Then $w\tilde{W}X_2-\Delta_{uv}w$ is a nonzero element of $N_1$ having length less than $w$.\\
    \textbf{Case III:} $\Lambda_{uv} = \Lambda_{xy}$ and $\Delta_{uv}=\Delta_{xy}$. Then we have 
    \[(rs)^{2u}\left(\lambda_4-\displaystyle\frac{s^2(r^2-s^2)}{1-rs^{-1}}\lambda_3\right)=(rs)^{2x}\left(\lambda_4-\displaystyle\frac{s^2(r^2-s^2)}{1-rs^{-1}}\lambda_3\right)\] From the assumption on the parameters we have 
    \begin{equation}\label{e1}
      (rs)^{2u}=(rs)^{2x}.  
    \end{equation}
     Now $\Lambda_{uv} = \Lambda_{xy}$ imply $(rs^{-1})^v=(rs^{-1})^y$. Since $0 \leq v,y \leq m_1-1$, we have $v=y$.\\
    If $\ord(s^2) \mid \ord(r^2s^2)\ord(rs^{-1})$, then $\ell_1=\ord(r^2s^2)$. Hence from Equation (\ref{e1}) we have $u=x$, which is a contradiction. \\
    If $\ord(s^2) \nmid \ord(r^2s^2)\ord(rs^{-1})$, then $\ell_1=2\ord(r^2s^2)$. Hence from Equation (\ref{e1}) we have $u-x=\ord(r^2s^2)$. Then we claim $\Psi_{uv}\neq \Psi_{xy}$. If not then we have $s^{2(u-x)m_1}=1$ which  implies that $\ord(s^2)$ divides $ (u-x)\ord(rs^{-1})=\ord(r^2s^2)\ord(rs^{-1})$, which is a contradiction.\\
    So in this case $wX_2^{m_1}-\Psi_{uv}w$ is a nonzero element of $N_1$ having length less than $w$.
    
    \noindent By repeating this process, we must eventually obtain an element of length one in $N_1$, i.e., a single basis vector $e(a,b)$. Finally, since the action of the generators $X_1$, $X_2$, $X_3$, and $\tilde{W}$ on $e(a,b)$ produces all other basis vectors of $\mathcal{M}_1(\lambda)$, it follows that $N_1$ must equal the whole module $\mathcal{M}_1(\lambda)$. Thus, $\mathcal{M}_1(\lambda)$ is simple.
\end{proof}
   \begin{remark}
It is important to note that if $\ord(rs)$ is odd, then $\ord(s^2)$ divides $\ord(r^2s^2)\ord(rs^{-1})$. So the dimension of the simple module will not exceed the PI degree. We will see that the simple modules classified in the subsequent cases also satisfy this property.
\end{remark}
\subsubsection{\texorpdfstring{\textbf{\upshape Simple $\mathcal{B}$-modules of type $\mathcal{M}_2(\mu)$}}{}}
 Let $ \mu = (\mu_1, \mu_2, \mu_3) \in (\mathbb{C}^*)^2 \times \mathbb{C} $. 
Define $ \mathcal{M}_2(\mu) $ as a vector space with basis  
$$
    \{ e(a, b) \mid 0 \leq a \leq \ell_2 - 1, \ 0 \leq b \leq m_2 - 1 \},
$$  
where $ \ell_2, m_2 \in \mathbb{N} $ are given by  
\begin{equation}\label{Eq5.6}
  \ell_2 := \ord(r^{2}s^2), \quad \text{and}\quad   m_2 := 
    \begin{cases}
        \ord(rs^{-1}), & \text{if } \ord(r^2) \mid \ord(r^{2}s^2) \ord(rs^{-1}) \text{ or } \mu_3=0 , \\[5pt]
        2\ord(rs^{-1}), & \text{otherwise}.
    \end{cases}
\end{equation}
We define the $ \mathcal{B} $-module structure on $ \mathcal{M}_2(\mu) $ by
\begin{align}\label{Eq5.7}
e(a,b)X_1 &= r^{-2b} \mu_1 e(a \oplus 1, b),\quad \end{align}
\begin{align}
 e(a,b)X_2&= 
    \begin{cases}\label{Eq5.8}
    \displaystyle(rs)^{2a}s^2 (r^2 - s^2)  
            \frac{1 - (r^{-1}s)^b}{1 - r^{-1}s} \mu_2 e(a,b-1), & b \neq 0, \\[5pt]
        0, & b = 0,
    \end{cases}\\ \label{Eq5.9}
e(a,b)X_3 &= (rs)^{2a+b} \mu_1^{-1} \mu_2 e(a \oplus (-1), b), \\ e(a,b)\tilde{W}&=\begin{cases}
    e(a,b+1), & b \neq m_2-1,\\
    r^{2am_2}\mu_3e(a,0), & b=m_2-1,
\end{cases}\notag
\end{align}
where $\oplus$ denotes the addition modulo $\ell_2$. Similarly, as above, we show that these actions indeed define a $\mathcal{B}$-module. \\
For $b \neq 0, m_2-1$, we have 
\[e(a,b)\tilde{W}X_2=e(a,b+1)X_2=(rs)^{2a}s^2(r^2-s^2)\displaystyle\frac{1 - (r^{-1}s)^{b+1}}{1 - r^{-1}s} \mu_2e(a,b). \]
Again, 
\begin{align*}
    &e(a,b)\left(X_2\tilde{W}+s^2(r^2-s^2)X_3X_1\right)\\
    &=(rs)^{2a}s^2 (r^2 - s^2)  
           \displaystyle \frac{1 - (r^{-1}s)^b}{1 - r^{-1}s} \mu_2 e(a,b-1)\tilde{W}+s^2(r^2-s^2)(rs)^{2a+b}r^{-2b}\mu_2 e(a,b)\\
     &=(rs)^{2a}s^2 (r^2 - s^2)  
           \displaystyle \frac{1 - (r^{-1}s)^b}{1 - r^{-1}s} \mu_2 e(a,b)\tilde{W}+s^2(r^2-s^2)(rs)^{2a+b}r^{-2b}\mu_2 e(a,b)\\
     &=(rs)^{2a}s^2(r^2-s^2)\mu_2\left(\displaystyle\frac{1 - (r^{-1}s)^b}{1 - r^{-1}s}+(r^{-1}s)^b\right)e(a,b)\\
     &=(rs)^{2a}s^2(r^2-s^2)\displaystyle\frac{1 - (r^{-1}s)^{b+1}}{1 - r^{-1}s} \mu_2e(a,b).
\end{align*}
For $b=0$, we have 
\[e(a,0)\tilde{W}X_2=e(a,1)X_2=(rs)^{2a}s^2(r^2-s^2)\mu_2e(a,0).\]
Also 
\[e(a,0)\left(X_2\tilde{W}+s^2(r^2-s^2)X_3X_1\right)=(rs)^{2a}s^2(r^2-s^2)\mu_2e(a,0)\]
For $b=m_2-1$, we have \[e(a,m_2-1)\tilde{W}X_2=r^{2am_2}\mu_3e(a,0)X_2=0\]
Finally,
\begin{align*}
    &e(a,m_2-1)\left(X_2\tilde{W}+s^2(r^2-s^2)X_3X_1\right)\\
    &=(rs)^{2a}s^2 (r^2 - s^2)  
           \displaystyle \frac{1 - (r^{-1}s)^{m_2-1}}{1 - r^{-1}s} \mu_2 e(a,b-1)\tilde{W} \\
          &\quad+s^2(r^2-s^2)(rs)^{2a+m_2-1}r^{-2(m_2-1)}\mu_2 e(a,b)\\
     &=-(rs)^{2a}s^2 (r^2 - s^2) rs^{-1} \mu_2 e(a,b)\tilde{W}+s^2(r^2-s^2)(rs)^{2a}(r^{-1}s)^{m_2-1}\mu_2 e(a,b)\\
     &=-(rs)^{2a}s^2 (r^2 - s^2) rs^{-1} \mu_2 e(a,b)\tilde{W}+s^2(r^2-s^2)(rs)^{2a}rs^{-1}\mu_2 e(a,b)=0. 
\end{align*}
The other relations can easily be verified. Then we have the following theorem
\begin{thm}\label{dim2}
    $\mathcal{M}_2(\mu)$ is a simple $\mathcal{B}$-module of dimension $\ell_2 m_2$.
\end{thm}
\begin{proof}
    The proof is parallel to the proof of Theorem \ref{dim1}. Let $N_1$ be a nonzero submodule of $\mathcal{M}_2(\mu)$. We claim that $N_1$ contains at least one basis vector. To see this, we define the length of a vector $v \in \mathcal{M}_2(\mu)$ with respect to the basis $\{e(a,b): 0 \leq a \leq \ell_2-1, 0 \leq b \leq m_2-1\}$ as follows: if $v=\sum_{i=1}^{n}\lambda_{i}e(a^{(i)},b^{(i)})$ where $\lambda_i \in \mathbb{C}^*$, and the basis elements $e(a^{(i)},b^{(i)})$ appearing in the sum are pairwise distinct then we define the length of $v$ to be $n$.\\
    Let $w=\sum_{\text{finite}}\lambda_{ab}e(a,b)$ be an element in $N_1$ where $\lambda_{ab} \in \mathbb{C}^*$. If $w$ has length $1$ then we are done. If not then there exists distinct pairs $(u,v)$ and $(x,y)$ with $0 \leq u,x \leq \ell_2-1$ and $0 \leq v,y \leq m_2-1$ such that $\lambda_{uv},\lambda_{xy} \neq 0$.
    Note that 
    \begin{align*}
      e(a,b)&X_2\tilde{W}=s^2(r^2-s^2)(rs)^{2a}\displaystyle\frac{1-(r^{-1}s)^b}{1-r^{-1}s}\mu_2e(a,b), \\
        e(a,b)X_3X_1&=(rs)^{2a}(r^{-1}s)^b \mu_2 e(a,b),  \quad e(a,b)X_1^{\ell_2}=r^{-2b\ell_2}\mu_1^le(a,b).
    \end{align*}
    Define \[\Lambda_{ab}:=(rs)^{2a}(r^{-1}s)^b \mu_2,\quad \Delta_{ab}:=s^2(r^2-s^2)(rs)^{2a}\displaystyle\frac{1-(r^{-1}s)^b}{1-r^{-1}s}\mu_2,\quad \Psi_{ab}:=r^{-2b\ell_2}\mu_1^l.\] Then we have the following cases:\\
    \textbf{Case I:} $\Lambda_{uv} \neq \Lambda_{xy}$. Then $e(u,v)$ and $e(x,y)$ are eigenvectors of the operator $X_3X_1$ corresponding to distinct eigenvalues. Then $wX_3X_1-\Lambda_{uv}w$ is a nonzero element of $N_1$ having length less than $w$.\\
    \textbf{Case II:} $\Lambda_{uv} = \Lambda_{xy}$ and $\Delta_{uv}\neq\Delta_{xy}$. Then $e(u,v)$ and $e(x,y)$ are eigenvectors of the operator $\tilde{W}X_2$ corresponding to distinct eigenvalues. Then $wX_2\tilde{W}-\Delta_{uv}w$ is a nonzero element of $N_1$ having length less than $w$.\\
    \textbf{Case III:} $\Lambda_{uv} = \Lambda_{xy}$ and $\Delta_{uv}=\Delta_{xy}$. Then we have  $(rs)^{2u}=(rs)^{2x}$. Since $0 \leq u,x \leq \ord(r^2s^2)-1$, we have $u=x$.
     Now $\Lambda_{uv} = \Lambda_{xy}$ imply 
     \begin{equation}\label{e2}
        (r^{-1}s)^v=(r^{-1}s)^y. 
     \end{equation} Since $0 \leq v,y \leq m_2-1$, we have $\ord(r^{-1}s)\mid v-y$.\\
    If $\ord(r^2) \mid \ord(r^2s^2)\ord(rs^{-1})$, then $m_2=\ord(rs^{-1})$. Hence from Equation (\ref{e2}), we have $u=x$, which is a contradiction. \\
    If $\ord(r^2) \nmid \ord(r^2s^2)\ord(rs^{-1})$, then $m_2=2\ord(rs^{-1})$. Hence from Equation (\ref{e2}) we have $v-y=\ord(r^{-1}s)$. Then we claim $\Psi_{uv}\neq \Psi_{xy}$. If not then we have $r^{2(v-y)\ell_2}=1$ which implies that $\ord(r^2)\mid (v-y)\ord(r^2s^{2})=\ord(rs^{-1})\ord(r^2s^2)$, which is a contradiction.\\
    So in this case $wX_1^{\ell_2}-\Psi_{uv}w$ is a nonzero element of $N_1$ having length less than $w$. 
    
    \noindent By repeating this process, we must eventually obtain an element of length one in $N_1$, i.e., a single basis vector $e(a,b)$. Finally, since the action of the generators $X_1$, $X_2$, $X_3$, and $\tilde{W}$ on $e(a,b)$ produces all other basis vectors of $\mathcal{M}_2(\lambda)$, it follows that $N_1$ must equal the whole module $\mathcal{M}_2(\lambda)$. Thus, $\mathcal{M}_2(\lambda)$ is simple.
\end{proof}

 \subsection{\texorpdfstring{ Construction of $X_1$-torsionfree simple $\mathcal{B}$-modules of Type II}{}}
 In this subsection we concentrate on construction of Type II $X_1$-torsionfree simple $\mathcal{B}$-module. From Remark \ref{rem1}, it is quite clear that we are constructing simple modules over $\mathcal{\tilde{B}}$.\\
For $ \epsilon = (\epsilon_1, \epsilon_2, \epsilon_3) \in (\mathbb{C}^*)^3$, define $ \mathcal{M}_3(\epsilon) $ as a vector space with basis  
$$
    \{ e(a) \mid 0 \leq a \leq \ell - 1\},
$$ where $\ell=\operatorname{lcm}\left(\ord(rs),\ord(rs^{-1})\right)$. We define the $ \mathcal{B} $-module structure on $ \mathcal{M}_3(\epsilon) $ by defining the actions :
\begin{align}\label{Eq5.11}
e(a)X_1&=\epsilon_1^2\epsilon_3\epsilon_2^{-1}\displaystyle\frac{1-rs^{-1}}{s^2(r^2-s^2)}s^{-2a}e(a\oplus2),\\ \label{Eq5.12}
e(a)X_2&=\epsilon_1e(a\oplus1),\\ \label{Eq5.13}
    e(a)X_3&=(rs)^a\epsilon_2e(a),\\
    e(a)\tilde{W}&=\epsilon_1\epsilon_3(rs^{-1})^ae(a\oplus1)\notag,
\end{align} where $\oplus$ is the addition modulo $\ell$. We can easily check the defining relations. Finally, we have the following theorem
\begin{thm}\label{dim3}
    $\mathcal{M}_3(\epsilon)$ is a simple $\mathcal{B}$-module of dimension $\ell$.
\end{thm}
\begin{proof}
    The proof is parallel to the proof of Theorem \ref{dim1}. Let $N_1$ be a nonzero submodule of $\mathcal{M}_3(\epsilon)$. We claim that $N_1$ contains a basis vector. We define the length of a vector $v \in \mathcal{M}_3(\epsilon)$ with respect to the basis $\{e(a): 0 \leq a \leq \ell-1\}$ as follows: if $v=\sum_{i=1}^{n}\lambda_{i}e(a^{(i)})$ where $\lambda_i \in \mathbb{C}^*$, and the basis elements $e(a^{(i)})$ appearing in the sum are pairwise distinct then we define the length of $v$ to be $n$.\\
    Let $w=\sum_{\text{finite}}\lambda_{a}e(a)$ be an element in $N_1$ where $\lambda_{a} \in \mathbb{C}^*$. If $w$ has length $1$ then we are done. If not then there exists distinct pairs $u$ and $x$ with $0 \leq u,x \leq \ell-1$ such that $\lambda_{u},\lambda_{x} \neq 0$. Observe that $X_2$ is an invertible operator on $\mathcal{M}_3(\epsilon)$. Moreover we have
    \begin{align*}
        e(a)X_3=(rs)^a\epsilon_2e(a), \quad e(a)\tilde{W}X_2^{-1}=\epsilon_3(rs^{-1})^ae(a).
    \end{align*}
    Define $\Lambda_a:=(rs)^a\epsilon_2$ and $\Delta_a:=(rs^{-1})^a\epsilon_3$. Then we have the following cases:\\
    \textbf{Case I:} $\Lambda_a\neq \Lambda_x$. Then $e(u)$ and $e(x)$ are eigenvectors of the operator $X_3$ corresponding to distinct eigenvalues. Then $wX_3-\Lambda_{u}w$ is a nonzero element of $N_1$ having length less than $w$.\\
    \textbf{Case II:} $\Lambda_a= \Lambda_x$ and $\Delta_a\neq \Delta_x$. Then $e(u)$ and $e(x)$ are eigenvectors of the operator $\tilde{W}X_2^{-1}$ corresponding to distinct eigenvalues. Then $w\left(\tilde{W}X_2^{-1}\right)-\Delta_{u}w$ is a nonzero element of $N_1$ having length less than $w$.\\
    \textbf{Case III:}  $\Lambda_a= \Lambda_x$ and $\Delta_a =\Delta_x$. Then we have $\ord(rs)\mid (a-x)$ and $\ord(rs^{-1})\mid (a-x)$ implying that $\ell \mid (a-x)$, a contradiction. Hence this case is impossible. \\
   \noindent By repeating this process, we must eventually obtain an element of length one in $N_1$, i.e., a single basis vector $e(a)$. Finally, since the action of the generators $X_1$, $X_2$, $X_3$, and $\tilde{W}$ on $e(a)$ produces all other basis vectors of $\mathcal{M}_3(\epsilon)$, it follows that $N_1$ must equal the whole module $\mathcal{M}_3(\epsilon)$. Thus, $\mathcal{M}_3(\epsilon)$ is simple.
\end{proof}
\section{\texorpdfstring{Classification of $X_1$-torsionfree simple $\mathcal{B}$-modules}{}}
This section is devoted to the complete classification of all $ X_1$-torsionfree simple modules over the subalgebra $ \mathcal{B} $. Building on the explicit constructions in Section \ref{Sec5}, we analyze the action of central elements to distinguish non-isomorphic modules. The classification is carried out by examining eigenvalues arising from the action of these elements and by identifying the parameter spaces that uniquely determine the isomorphism classes of simple $ \mathcal{B}  $-modules in this category.

\subsection{\texorpdfstring{Classification of $X_1$-torsionfree simple $\mathcal{B}$-moduls of Type I}{} }
Let $\mathcal{N}$ be a $\tilde{X}$-torsionfree simple $\B$-module with invertible action of $X_1$. Then by Proposition \ref{pidimresult}, the $\mathbb{K}$-dimension of $\mathcal{N}$ is finite and does not exceed $\pideg\left(\B\right)$. Note that $X_2^l$ is a central element of $\B$ and hence by Schur's lemma acts as a scalar, say $\alpha$, on $\mathcal{N}$. Depending on whether $\alpha$ is zero or nonzero we have the following two cases.\\
\textbf{Case I:} Let $\alpha \neq 0$. Note that the elements 
\begin{equation}\label{e3}
    X_1^{\ell_1}, \quad X_2^{m_1},\quad X_3X_1, \quad \tilde{W}X_2
\end{equation} are commuting elements in $\B$, where $ \ell_1, m_1 \in \mathbb{N} $ are as in Equation (\ref{Eq5.1}).  

 \noindent Since $\mathcal{N}$ is finite-dimensional, these commuting operators have a common eigenvector $v\in\mathcal{N}$. Thus there exist scalars $\eta_1,\eta_2,\eta_3,\eta_4\in\C$ such that
\begin{equation}\label{commeigen}
    vX_{1}^{\ell_1}=\eta_1v,\quad vX_{2}^{m_1}=\eta_2v, \quad vX_{3}X_{1}=\eta_3v, \quad v\tilde{W}X_2=\eta_4v.
\end{equation}
Since $X_1$ acts invertibly on $\mathcal{N}$, we have $\eta_1, \eta_3 \neq 0$. Since $\alpha \neq 0$, we have $\eta_2 \neq 0$. Hence each of the vectors $vX_1^{a}X_2^{b}$ with $0 \leq a \leq \ell_1-1$ and $0 \leq b \leq m_1-1$ is nonzero. Let us choose
\[\lambda_1:=\eta_1^{\frac{1}{\ell_1}},\quad \lambda_2:=\eta_2^{\frac{1}{m_1}}, \quad \lambda_3:=\eta_3,\quad \lambda_4:=\eta_4,\] so that
$\lambda=(\lambda_1,\lambda_2,\lambda_3,\lambda_4)\in {(\mathbb{C}^*)}^3 \times \mathbb{C}$. As $\mathcal{N}$ is $\tilde{X}$-torsionfree from the expression of $\tilde{X}$ in Equation (\ref{cen3}), we have 
\[\lambda_4-\displaystyle\frac{s^2(r^2-s^2)}{1-rs^{-1}}\lambda_3\neq 0.\]Now define a $\mathbb{K}$-linear map \[\Phi_1:\mathcal{M}_1(\lambda)\rightarrow \mathcal{N}\] by specifying the image of the basis vectors of $\mathcal{M}_1(\lambda)$ as follows:
\[\Phi_1\left(e(a,b)\right)=\lambda_1^{-a}\lambda_2^{-b}vX_{1}^aX_{2}^b,\ \ 0\leq a\leq \ell_1-1, 0 \leq b \leq m_1-1.\] We can easily verify that $\Phi_1$ is a nonzero $\B$-module homomorphism. In this verification, the following calculation is useful:
\begin{align*}
    (vX_{1}^aX_{2}^b)\tilde{W}&=\begin{cases}
    (rs)^{2a}\left(\eta_4-s^2(r^2-s^2)\displaystyle\frac{1-(rs^{-1})^b}{1-rs^{-1}}\eta_3\right)\left(vX_1^{a}X_2^{b-1}\right)&\text{if}\ b\neq 0\\
    \eta_2^{-1}\eta_4r^{2a}s^{-2a(m_1-1)}(vX_{1}^{a}X_{2}^{m_1-1})&\text{if}\ b= 0.
    \end{cases}    
\end{align*} Thus by Schur's lemma, $\Phi_1$ must be a $\B$-module isomorphism.

\noindent\textbf{Case II:} Let $\alpha=0$. Then $X_2$ acts nilpotently on the simple $\B$-module $\mathcal{N}$. Since the $\ker(X_2)$ is invariant under the commuting operators 
\[X_1^{\ell_2},\quad \tilde{W}^{m_2},\quad  X_3X_1.\] We can choose a common eigenvector $v \in \ker(X_2)$ such that 
\[vX_1^{\ell_2}=\zeta_1v,\quad v\tilde{W}^{m_2}=\zeta_2v, \quad vX_3X_1=\zeta_3v\] where $ \ell_2, m_2 \in \mathbb{N} $ are given by  
$$
  \ell_2 := \ord(r^{2}s^2), \quad \text{and}\quad   m_2 := 
    \begin{cases}
        \ord(rs^{-1}), & \text{if } \ord(r^2) \mid \ord(r^{2}s^2) \ord(rs^{-1}), \\[5pt]
        2\ord(rs^{-1}), & \text{otherwise}.
    \end{cases}
$$ Sine $X_1$ acts invertibly on $\mathcal{N}$, we have $\zeta_1, \zeta_3 \neq 0$.

\noindent If $\zeta_2 \neq 0$, then all the vectors $vX_1^{a}\tilde{W}^{b}$ with $0 \leq a \leq \ell_2-1$ and $0 \leq b \leq m_2-1$ are nonzero.\\
If $\zeta_2=0$, let $p$ be the smallest integer with $1 \leq p \leq m_2$ such that $v\tilde{W}^{p-1}\neq 0$ and $v\tilde{W}^p=0$. In this case, we claim that $p=\ord(rs^{-1})$. Indeed, consider $\left(v\tilde{W}^p\right)X_2=0$. This gives
\[0=v\left(X_2\tilde{W}^p+s^2(r^2-s^2)\displaystyle\frac{1-(r^{-1}s)^p}{1-r^{-1}s}X_3X_1\tilde{W}^{p-1}\right)=s^2(r^2-s^2)\displaystyle\frac{1-(r^{-1}s)^p}{1-r^{-1}s}\zeta_3(v\tilde{W}^{p-1}).\] This implies that $(rs^{-1})^p=1$. By the choice of $m_2$, we see that $p=\ord(rs^{-1})$ or $p=2\ord(rs^{-1})$.\\
Note if $\ord(r^2) \mid \ord(r^{2}s^2) \ord(rs^{-1})$, then we must have $p=\ord(rs^{-1})$. Let $\ord(r^2) \nmid \ord(r^{2}s^2) \ord(rs^{-1})$. Assume if $p=2\ord(rs^{-1})$, then the $\mathbb{K}$ linear span 
\[S:=\big<vX_{1}^a\tilde{W}^b:0 \leq a \leq \ell_2-1, \ord(rs^{-1})\leq b \leq 2\ord(rs^{-1})-1\big>\] is a nonzero submodule of $\mathcal{N}$. Since $\mathcal{N}$ is simple, therefore $S=\mathcal{N}$. In particular $v\in S$, so $v\tilde{W}^{\ord(rs^{-1})}=0$, which is a contradiction. Thus we define $$
  \ell_2 := \ord(r^{2}s^2), \quad \text{and}\quad   m_2 := 
    \begin{cases}
        \ord(rs^{-1}), & \text{if } \ord(r^2) \mid \ord(r^{2}s^2) \ord(rs^{-1}) \ \text{or} \  \zeta_2=0, \\[5pt]
        2\ord(rs^{-1}), & \text{otherwise}.
    \end{cases} 
$$ and the above argument ensures that all the vectors $vX_1^{a}\tilde{W}^{b}$ with $0 \leq a \leq \ell_2-1$ and $0 \leq b \leq m_2-1$ are nonzero. Let us choose
\[\mu_1:=\zeta_1^{\frac{1}{\ell_2}},\ \mu_2:=\zeta_3,\ \mu_3:=\zeta_2^{\frac{1}{m_2}},\] so that 
$\mu=(\mu_1,\mu_2,\mu_3)\in {(\mathbb{C}^*)}^2 \times \mathbb{C}$.  Now define a $\mathbb{K}$-linear map \[\Phi_2:\mathcal{M}_2(\mu)\rightarrow \mathcal{N}\] by specifying the image of the basis vectors of $\mathcal{M}_2(\mu)$ as follows:
\[\Phi_2\left(e(a,b)\right)=\mu_1^{-a}vX_{1}^a\tilde{W}^b,\ \ 0\leq a\leq \ell_2-1, 0 \leq b \leq m_2-1.\] We can easily verify that $\Phi_2$ is a nonzero $\B$-module homomorphism. In this verification, the following calculation will be useful.
\begin{align*}
    (vX_{1}^a\tilde{W}^b)X_2&=\begin{cases}
    (rs)^{2a}s^2(r^2-s^2)\displaystyle\frac{1-(r^{-1}s)^b}{1-r^{-1}s}\zeta_3\left(vX_1^{a}\tilde{W}^{b-1}\right)&\text{if}\ b\neq 0\\
    0&\text{if}\ b= 0.
    \end{cases}    
\end{align*} Thus by Schur's lemma, $\Phi_2$ must be a $\B$-module isomorphism.
\subsection{\texorpdfstring{Classification of $X_1$-torsionfree simple $\mathcal{B}$-moduls of Type II}{}} From Remark \ref{rem1}, it is clear that any simple $\tilde{X}$-torsion simple $\mathcal{B}$-module $\mathcal{N}$ becomes a simple $\mathcal{\tilde{B}}$-module. Since $\mathcal{N}$ is $X_1$-torsionfree and we have $ \tilde{W}X_2=\displaystyle\frac{s^2(r^2-s^2)}{1-rs^{-1}}X_3X_1$ in $\mathcal{\tilde{B}}$, $X_2$ and $\tilde{W}$ also act as invertible operators on $\mathcal{N}$. Consider the commutating operators 
\[X_2^{\ell}, \quad X_3,\quad \tilde{W}X_2^{-1}\] in $\mathcal{\tilde{B}}$ where $\ell=\operatorname{lcm}\left(\ord(rs),\ord(rs^{-1})\right)$. Since $\mathcal{\tilde{B}}$ is a polynomial identity algebra $\mathcal{N}$ is a finite dimensional simple $\mathcal{\tilde{B}}$-module. As $\mathcal{N}$ is finite dimensional we have $v(\neq 0)\in \mathcal{N}$ such that 
\[vX_2^{\ell}=\alpha_1v, \quad vX_3=\alpha_2v, \quad v\tilde{W}X_2^{-1}=\alpha_3v,\] for $\alpha_1,\alpha_2,\alpha_3 \neq 0$. As $\alpha_1 \neq 0$, each of the vectors $vX_2^{a}$ with $0 \leq a \leq \ell-1$ are all nonzero. Let us choose 
\[\epsilon_1=\alpha_1^{\frac{1}{\ell}}, \quad \epsilon_2=\alpha_2, \quad \epsilon_3=\alpha_3,\] so that 
$\epsilon=(\epsilon_1,\epsilon_2,\epsilon_3)\in {(\mathbb{C}^*)}^3 $.  Now define a $\mathbb{K}$-linear map \[\Phi_3:\mathcal{M}_3(\epsilon)\rightarrow \mathcal{N}\] by specifying the image of the basis vectors of $\mathcal{M}_3(\epsilon)$ as follows:
\[\Phi_3\left(e(a)\right)=\epsilon_1^{-a}vX_2^a,\ \ 0\leq a\leq \ell-1.\] We can easily verify that $\Phi_3$ is a nonzero $\mathcal{\tilde{B}}$-module homomorphism. By Schur's lemma $\Phi_3$ is a $\mathcal{\tilde{B}}$-module isomorphism.
\par The preceding discussion leads to one of the key results of this section, offering a framework for classifying simple 
$\B$-modules in terms of scalar parameters.
\begin{thm}\label{m3}
    Suppose $\mathcal{N}$ is a simple $X_1$-torsionfree $\B$-module. Then $\mathcal{N}$ is isomorphic to exactly one of the following simple $\B$-modules:
\begin{enumerate}[label= \((\arabic*)\)]
    \item $\mathcal{M}_1(\lambda)$ for some ${\lambda}=(\lambda_1,\lambda_2,\lambda_3, \lambda_4)\in \mathbb{({C}^*)}^3 \times \mathbb{C}$ if $\mathcal{N}$ is $\tilde{X}$ and $X_{2}$-torsionfree.
    \item $\mathcal{M}_2(\mu)$ for some ${\mu}=(\mu_1,\mu_2,\mu_3)\in \mathbb{({C}^*)}^2 \times \mathbb{C}$ if $\mathcal{N}$ is $\tilde{X}$-torsionfree and $X_{2}$-torsion.
    \item $\mathcal{M}_3(\epsilon)$ for some $\epsilon=(\epsilon_1, \epsilon_2, \epsilon_3) \in (\mathbb{C}^{*})^3$  if  $\mathcal{N}$ is $\tilde{X}$-torsion.
\end{enumerate}
\end{thm}
\section{\texorpdfstring{$X_1$-torsionfree Simple $ U_{r,s}^+(B_2) $-Modules}{}}

In this section, we study how the $ X_1 $-torsionfree simple modules over the subalgebra $ \mathcal{B} $ can be extended to simple modules over the full algebra $ U^+_{r,s}(B_2) $. Using the invertibility of $X_1$, we explicitly define the action of the remaining generator $X_4$, thereby obtaining a well-defined $ U^+_{r,s}(B_2) $-module structure. We use the one-to-one correspondence between simple $ \mathcal{B}  $-modules with invertible $ X_1 $-action and simple $ U^+_{r,s}(B_2) $-modules under the same condition (Theorem \ref{itd}), thus lifting the classification from $ \mathcal{B} $ to the full algebra. This gives us the following three non-isomorphic simple $X_1$-torsionfree $U_{r,s}^+(B_2)$-modules.
\subsection{\texorpdfstring{Simple $ U_{r,s}^+(B_2) $-modules of type $ \mathcal{M}(\lambda) $}{}}

Let $ \lambda = (\lambda_1, \lambda_2, \lambda_3, \lambda_4) \in (\mathbb{C}^*)^3 \times \mathbb{C} $ be a parameter satisfying  
\[\lambda_4 \neq \frac{s^2 (r^2 - s^2)}{1 - rs^{-1}} \lambda_3.\]
Define $ \mathcal{M}(\lambda) $ as a vector space with basis  
$$
    \{ e(a, b) \mid 0 \leq a \leq \ell_1 - 1, \ 0 \leq b \leq m_1 - 1 \},
$$  
where $ \ell_1, m_1 \in \mathbb{N} $ are as in Equation (\ref{Eq5.1}). \\
We define the $ U_{r,s}^+(B_2) $-module structure on $ \mathcal{M}(\lambda) $ by defining the actions of $X_1, X_2, X_3$ as in Equations (\ref{Eq5.2}), (\ref{Eq5.3}), (\ref{Eq5.4}), respectively, and defining the action of $X_4$ as follows:
\begin{align*}
    e(a,b)X_4 &= 
    \begin{cases}
        \displaystyle \frac{1}{\Delta} \Bigg[ 
            (rs)^{2a} s^{2(b-1)} T_bE^{b-1}_{a,-1}- s^{2(b+1)} \lambda_1^{-1} \lambda_2   E^{b+1}_{a,-1}
        \Bigg],  
        &  b \neq 0, m_1 - 1, \\[15pt]
\displaystyle\frac{1}{\Delta} \Bigg[ 
            (rs)^{2a} s^{-2(a-1) m_1+1} T_0E^{m_1-1}_{a,-1} 
            - s^{2} \lambda_1^{-1} \lambda_2 E^{1}_{a,-1} 
        \Bigg],  
        &  b = 0, \\[15pt]
\displaystyle\frac{1}{\Delta} \Bigg[ 
            (rs)^{2a} s^{2(m_1 - 2)} T_{m_1-1}
            E^{m_1-2}_{a,-1} - s^{2 a m_1} \lambda_1^{-1} \lambda_2 E^{0}_{a,-1} 
        \Bigg],  
        & b = m_1 - 1,
    \end{cases}
\end{align*}
where, $$\Delta=r^2 - s^2,\quad T_b=\lambda_1^{-1} \lambda_2^{-1} 
            \bigg( \lambda_4 - s^2 \Delta  
            \frac{1 - (rs^{-1})^b}{1 - rs^{-1}} \lambda_3 \bigg), \quad E^{b}_{a,i}=e(a\oplus i,b)\quad (i\in\Z),$$ and $\oplus$ denotes the addition modulo $\ell_1$.

\begin{thm}
    The module structure of $ U_{r,s}^+(B_2) $ on $ \mathcal{M}(\lambda) $, as defined above, is well defined.
\end{thm}
\begin{proof}
    To prove that the module structure defined above is well defined, we need to show that the actions satisfy all the relations given in Equation (\ref{Eq2.1}).
\begin{align*}e(a,b)X_1X_2&= s^{-2b}e(a\oplus1,b)X_2=\begin{cases}
        \displaystyle\displaystyle  s^{-2b}\lambda_1\lambda_2e(a\oplus1,b+1),&b\neq m_1-1\\[5pt]
        \displaystyle s^{2(am_1+1)}e(a\oplus1,0),&b= m_1-1
    \end{cases}\\
    &=\begin{cases} \displaystyle s^2\lambda_2e(a,b+1)X_1,& b\neq m_1-1\\[5pt]
     \displaystyle s^2s^{2am_1}\lambda_2e(a,0)X_1,&b=m_1-1
    \end{cases}=e(a,b)s^2X_2X_1.\end{align*}
    \begin{align*}
      e(a,b)X_1X_3&= s^{-2b} \lambda_1 e(a \oplus 1, b)X_3=s^{-2b}(rs)^{2(a+1)+b}\lambda_3e(a,b)\\
&=r^2s^2(rs)^{2a+b}\lambda_1^{-1}\lambda_3e(a\oplus(-1),b)X_1=e(a,b)r^2s^2X_3X_1.
    \end{align*}
\begin{align*}
      e(a,b)X_1X_4&=  s^{-2b} \lambda_1 e(a \oplus 1, b)X_4\\[5pt] &=\begin{cases}
        \displaystyle \frac{\lambda_1}{\Delta} \Bigg[ 
            (rs)^{2(a+1)} s^{-2} T_bE^{b-1}_{a,0}  - s^{2} \lambda_2 E^{b+1}_{a,0}  
        \Bigg], 
        &  b \neq 0, m_1 - 1 \\[15pt]
\displaystyle\frac{\lambda_1}{\Delta} \Bigg[ 
            (rs)^{2(a+1)} s^{-2 (a m_1+1)}  T_0 
           E^{m_1-1}_{a,0} - s^{2}  \lambda_2 E^{1}_{a,0}  
        \Bigg],  
        &  b = 0 \\[15pt]
\displaystyle\frac{\lambda_1}{\Delta} \Bigg[ 
            (rs)^{2(a+1)} s^{-2} T_{m_1-1} E^{m_1-2}_{a,0}- s^{2(am_1+1)}\lambda_2 E^{0}_{a,0}  
        \Bigg],  
        & b = m_1 - 1
\end{cases}\end{align*}
\begin{align*}
    &=\begin{cases}
        \displaystyle \frac{r^2}{\Delta} \Bigg[ 
            (rs)^{2a} s^{2(b-1)}T_b  
            E^{b-1}_{a,-1}- s^{2(b+1)} \lambda_1^{-1} \lambda_2 E^{b+1}_{a,-1}
        \Bigg]X_1 +\lambda_2E^{b+1}_{a,0},  
        &  b \neq 0, m_1 - 1 \\[15pt]
\displaystyle\frac{r^2}{\Delta} \Bigg[ 
            (rs)^{2a} s^{-2( (a-1) m_1+1)} T_0  E^{m_1-1}_{a,-1}
            - s^{2} \lambda_1^{-1} \lambda_2 E^{1}_{a,-1}
        \Bigg]X_1+\lambda_2E^{1}_{a,0},  
        &  b = 0 \\[15pt]
\displaystyle\frac{r^2}{\Delta} \Bigg[ 
            (rs)^{2a} s^{2(m_1 - 2)} T_{m_1-1} 
           E^{m_1-2}_{a,-1} - s^{2 a m_1} \lambda_1^{-1} \lambda_2 E^{0}_{a,-1} 
        \Bigg]X_1 +s^{2am_1}\lambda_2E^{0}_{a,0},  
        & b = m_1 - 1
    \end{cases}\\[3pt]
    &= e(a,b)(r^2X_4X_1+X_2).\\
     &e(a,b)X_2X_3 = 
    \begin{cases}
        \lambda_2 e(a,b+1)X_3=(rs)^{2a+b+1}\lambda_1^{-1}\lambda_2\lambda_3e(a\oplus(-1),b+1), & b \neq m_1 - 1 \\[3pt]
        s^{2 a m_1} \lambda_2 e(a,0)X_3=(rs)^{2a}s^{2am_1}\lambda_1^{-1}\lambda_2\lambda_3e(a\oplus(-1),0), & b = m_1 - 1
        \end{cases} \\[3pt]
        &\quad \quad \quad \quad \quad \,= 
    \begin{cases}(rs)^{2a+b+1}\lambda_1^{-1}\lambda_2\lambda_3e(a\oplus(-1),b+1), & b \neq m_1 - 1 \\[3pt]
       (rs)^{2a+b+1}s^{2(a-1)m_1}\lambda_1^{-1}\lambda_2\lambda_3e(a\oplus(-1),0), & b = m_1 - 1
        \end{cases}\\[3pt]
        &\quad \quad \quad \quad \quad \,=(rs)^{2a+b+1}\lambda_1^{-1}\lambda_2\lambda_3e(a\oplus(-1),b)X_2=e(a,b)rsX_3X_2.
        \end{align*}
       \begin{align*}
          &e(a,b)X_2X_4  = 
    \begin{cases}
        \lambda_2 e(a,b+1)X_4, & b \neq m_1 - 1 \\[3pt]
        s^{2 a m_1} \lambda_2 e(a,0)X_4, & b = m_1 - 1
    \end{cases}\\[3pt]
        &=\begin{cases}
        \displaystyle \frac{\lambda_2}{\Delta} \Bigg[ 
            (rs)^{2a} s^{2b} T_b 
            E^{b}_{a,-1}- s^{2(b+2)} \lambda_1^{-1} \lambda_2 E^{b+2}_{a,-1}  
        \Bigg],  
        &  b \neq m_1-2, m_1 - 1 \\[15pt]
\displaystyle\frac{\lambda_2}{\Delta} \Bigg[ 
            (rs)^{2a} s^{2(m_1 - 2)} T_{m_1-1}  
           E^{m_1-2}_{a,-1} - s^{2 a m_1} \lambda_1^{-1} \lambda_2 E^{0}_{a,-1}  
        \Bigg],  
        & b = m_1 - 2  \\[15pt]
\displaystyle\frac{\lambda_2}{\Delta} \Bigg[ 
            (rs)^{2a}  
            s^{2(m_1-1)} T_0E^{m_1-1}_{a,-1}  
            - s^{2(am_1+1)} \lambda_1^{-1} \lambda_2 E^{1}_{a,-1}  
        \Bigg],  
        &  b = m_1-1
    \end{cases}\\[3pt]
        &=\begin{cases}
    \displaystyle \frac{s^2}{\Delta} \Bigg[ 
            (rs)^{2a} s^{2(b-1)} T_b  
            \lambda_2E^{b}_{a,-1}- s^{2(b+1)} \lambda_1^{-1} \lambda_2^2 E^{b+2}_{a,-1}\Bigg]\\ \quad -s^2(rs)^{2a+b}\lambda_1^{-1}\lambda_3E^{b}_{a,-1},  
        &  b \neq m_1-2, m_1 - 1 \\[3pt]
\displaystyle\frac{s^2}{\Delta} \Bigg[ 
            (rs)^{2a} s^{2(m_1 - 3)} T_{m_1-2}
            \lambda_2E^{m_1-2}_{a,-1} - s^{2 (a m_1-1)} \lambda_1^{-1} \lambda_2^2 E^{0}_{a,-1}  
        \Bigg]\\ \quad -s^2(rs)^{2a+m_1-2}\lambda_1^{-1}\lambda_3E^{m_1-2}_{a,-1},  
        & b = m_1 - 2  \\[3pt]
\displaystyle\frac{s^2}{\Delta} \Bigg[ 
            (rs)^{2a}s^{2(m_1-2)}  T_{m_1-1} 
             \lambda_2E^{m_1-1}_{a,-1} 
- s^{2(am_1)} \lambda_1^{-1} \lambda_2^2 E^{1}_{a,-1}\Bigg]\\ \quad -s^2(rs)^{2a+m_1-1}\lambda_1^{-1}\lambda_3E^{m_1-1}_{a,-1},  
        &  b = m_1-1    
    \end{cases}\end{align*}
    \begin{align*}
   & =\begin{cases}
    \displaystyle \frac{s^2}{\Delta} \Bigg[ 
            (rs)^{2a} s^{2(b-1)} T_b  
            E^{b-1}_{a,-1}- s^{2(b+1)} \lambda_1^{-1} \lambda_2 E^{b+1}_{a,-1}\Bigg]X_2\\ \quad-s^2(rs)^{2a+b}\lambda_1^{-1}\lambda_3E^{b}_{a,-1},  
        &  b \neq 0, m_1 - 1 \\[7pt]
        \displaystyle \frac{s^2}{\Delta} \Bigg[ 
            (rs)^{2a} s^{-2((a-1)m+1)} T_0  
            E^{m_1-1}_{a,-1}- s^{2} \lambda_1^{-1} \lambda_2 E^{0}_{a,-1}\Bigg]X_2 \\ \quad-s^2(rs)^{2a}\lambda_1^{-1}\lambda_3E^{0}_{a,-1},  
        &  b = 0\\[7pt]
         \displaystyle \frac{s^2}{\Delta} \Bigg[ 
            (rs)^{2a} s^{2(m_1-2)} T_{m_1-1}  
            E^{m_1-2}_{a,-1}- s^{2am} \lambda_1^{-1} \lambda_2 E^{0}_{a,-1}\Bigg]X_2\\ \quad-s^2(rs)^{2a+m_1-1}\lambda_1^{-1}\lambda_3E^{m_1-1}_{a,-1},  
        &  b =m_1 - 1
        \end{cases}\\[7pt] &=e(a,b)(s^2X_4X_2-s^2X_3).\end{align*}
\begin{align*}
         e(a,b)&X_4X_3=\begin{cases}
        \displaystyle \frac{1}{\Delta} \Bigg[ 
            (rs)^{2a} s^{2(b-1)} T_b 
            E^{b-1}_{a,-1}- s^{2(b+1)} \lambda_1^{-1} \lambda_2 E^{b+1}_{a,-1} 
        \Bigg]X_3,  
        &  b \neq 0, m_1 - 1 \\[15pt]
\displaystyle\frac{1}{\Delta} \Bigg[ 
            (rs)^{2a} s^{-2 ((a-1) m_1+1)} 
            T_0 E^{m_1-1}_{a,-1} 
            - s^{2} \lambda_1^{-1} \lambda_2 E^{1}_{a,-1} 
        \Bigg]X_3,  
        &  b = 0 \\[15pt]
\displaystyle\frac{1}{\Delta} \Bigg[ 
            (rs)^{2a} s^{2(m_1 - 2)} T_{m_1-1}
            E^{m_1-2}_{a,-1} - s^{2 a m_1} \lambda_1^{-1} \lambda_2 E^{0}_{a,-1} 
        \Bigg] X_3,  
        & b = m_1 - 1
    \end{cases}\\[2pt]
         &=\begin{cases}
        \displaystyle \frac{\Omega}{\Delta} \Bigg[ 
            (rs)^{2(a-1)} s^{2(b-1)} T_b           E^{b-1}_{a,-2}- s^{2(b+1)} \lambda_1^{-1} \lambda_2 E^{b+1}_{a,-2} 
        \Bigg], 
        &  b \neq 0, m_1 - 1 \\[13pt]
\displaystyle\frac{\Omega}{\Delta} \Bigg[ 
            (rs)^{2(a-1)} s^{-2 (a-1) m_1} T_0  s^{2(m_1-1)} E^{m_1-1}_{a,-2}  
            - s^{2} \lambda_1^{-1} \lambda_2 E^{1}_{a,-2} 
        \Bigg],  
        &  b = 0 \\[13pt]
\displaystyle\frac{\Omega}{\Delta} \Bigg[ 
            (rs)^{2(a-1)} s^{2(m_1 - 2)} T_{m_1-1}
           E^{m_1-2}_{a,-2} - s^{2 (a-1) m_1} \lambda_1^{-1} \lambda_2 E^{0}_{a,-2}
        \Bigg],  
        & b = m_1 - 1
    \end{cases}\\[2pt]
    &=\Omega e(a\oplus(-1),b)X_4,\end{align*}
    $\text{where}\quad\Omega=(rs)^{2a+(b-1)}\lambda_1^{-1}\lambda_3,\quad \text{hence}\quad e(a,b)X_4X_3=(rs)^{-1}e(a,b)X_3X_4.$
\end{proof}

\subsection{\texorpdfstring{Simple $ U_{r,s}^+(B_2) $-modules of type $ \mathcal{M}(\mu) $}{}}

Let $ \mu = (\mu_1, \mu_2, \mu_3) \in (\mathbb{C}^*)^2 \times \mathbb{C} $. 
Define $ \mathcal{M}(\mu) $ as a vector space with basis  
$$
    \{ e(a, b) \mid 0 \leq a \leq \ell_2 - 1, \ 0 \leq b \leq m_2 - 1 \},
$$  
where $ \ell_2, m_2 \in \mathbb{N} $ are as in Equation (\ref{Eq5.6}).  

\noindent We define the $ U_{r,s}^+(B_2) $-module structure on $ \mathcal{M}(\mu) $ by defining the actions of $X_1, X_2, X_3$ as in Equations (\ref{Eq5.7}), (\ref{Eq5.8}), (\ref{Eq5.9}), respectively, and defining the action of $X_4$ as follows:
    
    \begin{align*}
    e(a,b)X_4 &= 
     \begin{cases}
        \displaystyle \displaystyle\frac{1}{r^2 - s^2} \Bigg[ 
            r^{2(b+1)} \mu_1^{-1}e(a \oplus (-1), b+1)-(rs)^{2a}s^2\\ \quad\cdot (r^2 - s^2)r^{2(b-1)}
            \displaystyle
            \frac{1 - (r^{-1}s)^b}{1 - r^{-1}s}\mu_1^{-1}\mu_2 e(a \oplus (-1), b-1)  
        \Bigg],  
        &  b \neq 0, m_2 - 1, \\
 \displaystyle\frac{1}{r^2 - s^2} r^{2(b+1)} \mu_1^{-1}e(a \oplus (-1), b+1),  
        &  b = 0, \\
 \displaystyle\frac{1}{r^2 - s^2} \Bigg[ 
            r^{2am_2} \mu_1^{-1} \mu_3e(a \oplus (-1), 0)  -(rs)^{2a}s^2 \\ \quad  \cdot \displaystyle(r^2 - s^2)r^{2(b-1)}  
            \frac{1 - (r^{-1}s)^b}{1 - r^{-1}s}\mu_1^{-1}\mu_2 e(a \oplus (-1), b-1)  
        \Bigg],  
        & b = m_2 - 1,
    \end{cases}
\end{align*}
where $\oplus$ denotes the addition modulo $\ell_2$.
\subsection{\texorpdfstring{Simple $ U_{r,s}^+(B_2) $-modules of type $ \mathcal{M}(\epsilon) $}{}} Let $ \epsilon = (\epsilon_1, \epsilon_2, \epsilon_3) \in (\mathbb{C}^*)^3 $. 
Define $ \mathcal{M}(\epsilon) $ as a vector space with basis  
$$
    \{ e(a) \mid 0 \leq a \leq \ell - 1\},
$$  
where $ \ell \in \mathbb{N} $ is given by $\ell:= \operatorname{lcm}(\ord(rs),\ord(rs^{-1}))$. 

\noindent We define the $ U_{r,s}^+(B_2) $-module structure on $ \mathcal{M}(\epsilon) $ by defining the actions of $X_1, X_2, X_3$ as in Equations (\ref{Eq5.11}), (\ref{Eq5.12}), (\ref{Eq5.13}), respectively, and defining the action of $X_4$ as follows:
\begin{align*}
e(a)X_4&=\epsilon_1^{-1}\epsilon_2\epsilon_3^{-1}\displaystyle\frac{(\epsilon_3\left(rs^{-1})^a-1\right)s^{2a}}{1-rs^{-1}} e(a\oplus(-1)),
\end{align*} where $\oplus$ denotes the addition modulo $\ell$.
\section{\texorpdfstring{Construction of $X_1$-torsion simple $ U_{r,s}^+(B_2) $-modules}{}}
In this section, we mainly focus on constructing $X_1$-torsion simple $ U_{r,s}^+(B_2) $-modules.
\subsection{\texorpdfstring{Simple $ U_{r,s}^+(B_2) $-modules of type $ \mathcal{M}(\nu) $}{}}

Let $ \nu = (\nu_1, \nu_2, \nu_3) \in (\mathbb{C}^*)^2 \times \mathbb{C} $. 
Define $ \mathcal{M}(\nu) $ as a vector space with basis  
$$
    \{ e(a, b) \mid 0 \leq a \leq \ell_3 - 1, \ 0 \leq b \leq m_3 - 1 \},
$$  
where $ \ell_3, m_3 \in \mathbb{N} $ are given by  
$$
  \ell_3 := \ord(rs), \quad \text{and}\quad   m_3 := 
    \begin{cases}
        \ord(r^{-2}s^2), & \text{if } \ord(s^2) \mid \ord(r^{-2}s^2) \ord(rs) \text{ or } \nu_3=0 , \\[5pt]
        2\ord(r^{-2}s^2), & \text{otherwise}.
    \end{cases}
$$
We define the $ U_{r,s}^+(B_2) $-module structure on $ \mathcal{M}(\nu) $ by defining the actions:
 $$ e(a,b)X_1 = 
    \begin{cases}
         -\displaystyle  \displaystyle \displaystyle s^{-4(b-1)}(rs)^{-2} (1 - r^{-1}s)  
            \frac{1 - (r^{-1}s)^{2b}}{1 - (r^{-1}s)^2} \nu_1^2 e(a\oplus 2,b-1), & b \neq 0, \\[10pt]
        0, & b = 0,
    \end{cases}$$
    $$ e(a,b)X_2 = s^{-2b} \nu_1 e(a \oplus 1, b),\quad e(a,b)X_3 = (rs)^{a} \nu_2 e(a, b),$$
   $$e(a,b)X_4 = 
    \begin{cases}
         \displaystyle \displaystyle\frac{s^{2b}}{s^{-2} - r^{-1}s^{-1}} \Bigg[ 
             \nu_1^{-1}e(a \oplus (-1), b+1)  
           \\ \quad
           
          -(rs)^{a-1}\nu_1^{-1}\nu_2 e(a \oplus (-1), b)  
        \Bigg],  
        &  b \neq  m_3 - 1,\\[15pt]
 \displaystyle\frac{s^{2b}}{s^{-2} - r^{-1}s^{-1}} \Bigg[ 
            s^{2(a-1)m_3} \nu_1^{-1} \nu_3e(a \oplus (-1), 0) 
          \\ \quad
          -(rs)^{a-1}  \nu_1^{-1} \nu_2 e(a \oplus (-1), b)  
        \Bigg],  
        & b = m_3 - 1,
    \end{cases}$$
where $\oplus$ denotes the addition modulo $\ell_3$. We can readily confirm the action above indeed defines an $U_{r,s}^+(B_2)$-module structure on $\mathcal{M}(\nu)$, similarly detailed in the first case.
\begin{theorem}
    $\mathcal{M}(\nu)$ is a simple $U_{r,s}^+(B_2)$-module of dimension $\ell_3m_3$.
\end{theorem}
\begin{proof}
     Observe that $X_2$ acts as an invertible operator on $\mathcal{M}(\nu)$. Then we have the following
    \begin{align*}
        e(a,b)X_3&=(rs)^a\nu_2e(a,b),\\
        e(a,b)X_1WX_2^{-2}&=-(rs)^{-2}(1-r^{-1}s)\displaystyle\frac{1-(r^{-1}s)^{2b}}{1-(r^{-1}s)^2}\nu_1^2e(a,b),\\
        e(a,b)X_2^{\ell_3}&=s^{-2b\ell_3}\nu_1^{\ell_3}e(a,b),
    \end{align*} where $W:=X_3+(s^{-2}-r^{-1}s^{-1})X_2X_4$.
    With this fact, the proof is parallel to the proof of Theorem \ref{dim1}.
\end{proof}
\subsection{\texorpdfstring{Simple $ U_{r,s}^+(B_2) $-modules of type $ \mathcal{M}(\xi) $}{}}

Let $ \xi = (\xi_1, \xi_2) \in \mathbb{C}^*\times \mathbb{C} $. 
Define $ \mathcal{M}(\xi) $ as a vector space with basis  
$$
    \{ e(a) \mid 0 \leq a \leq \ell_4 - 1\},
$$  
where $ \ell_4\in \mathbb{N} $ are given by  
$$
  \ell_4 := 
    \begin{cases}
        \operatorname{lcm}(\ord(rs),\ord(rs^{-1})), & \text{if } \xi_2\neq 0 , \\[5pt]
        \ord(rs^{-1}), & \text{if } \xi_2=0,
    \end{cases}
$$
We define the $ U_{r,s}^+(B_2) $-module structure on $ \mathcal{M}(\xi) $ by defining the actions:
 $$ e(a)X_1 = 
    \begin{cases}
       \displaystyle -r^{-2(a-1)}\frac{1 - (rs^{-1})^{a-1}}{1 - rs^{-1}}\frac{1 - (rs^{-1})^a}{1 - (rs^{-1})^2}\xi_1e(a-2), & a \neq 0,1, \\[5pt] 
        0, & a = 0,1,
    \end{cases}$$
    $$ e(a)X_2 = 
    \begin{cases}
       \displaystyle (rs)^{-(a-1)}\frac{1 - (rs^{-1})^{a}}{1 - rs^{-1}}\xi_1e(a-1), & a \neq 0, \\[5pt]
        0, & a = 0,
    \end{cases}$$
$$ e(a)X_3 = (rs)^{-a} \xi_1 e(a),\quad e(a)X_4 = 
    \begin{cases}
        e(a+1),  
        &  a \neq  \ell_4 - 1,\\[5pt]
\xi_2e(0),  
        & a = \ell_4 - 1.
    \end{cases} $$ We can easily check these actions indeed define a $U_{r,s}^{+}(B_2)$-module structure on $\mathcal{M}(\xi)$.
   \begin{thm}
    $\mathcal{M}(\xi) $ is a simple $U_{r,s}^{+}(B_2)$-module of dimension $\ell_4$.
   \end{thm}
\begin{proof}
   If $\ell_4= \operatorname{lcm}(\ord(rs),\ord(rs^{-1}))$, we have 
   \begin{align*}
       e(a)X_3&=(rs)^{-a}\xi_1e(a)\\
       e(a)X_2X_4&=(rs)^{-(a-1)}\displaystyle\frac{1-(rs^{-1})^a}{1-rs^{-1}}\xi_1e(a).
   \end{align*}
   If $\ell_4=\ord(rs^{-1})$, we have 
   \[e(a)X_2X_4X_3^{-1}=rs\displaystyle\frac{1-(rs^{-1})^a}{1-rs^{-1}}e(a).\]With these observations the proof is parallel to Theorem \ref{dim1}.
\end{proof}

 \section{\texorpdfstring{Classification of $X_1$-torsion Simple $U_{r,s}^{+}(B_2)$-modules}{}}
 In this section, we classify all simple $U_{r,s}^{+}(B_2)$-modules with nilpotent action of $X_1$, under the assumption that $r$ and $s$ are primitive $m$-th and $n$-th roots of unity, respectively. Let $\mathcal{N}$ represent a simple $U_{r,s}^{+}(B_2)$-module with nilpotent action of $X_1$. Note that $X_2^l$ is a central element of $U_{r,s}^{+}(B_2)$ and hence by Schur's lemma acts as a scalar, say $\beta$, on $\mathcal{N}$. Depending on whether $\beta$ is zero or nonzero we have the following two cases.\\
 \textbf{Case I:} Let $\beta \neq 0$. Since the $\ker(X_1)$ is invariant under the commuting operators 
 \[X_2^{\ell_3}, \quad W^{m_3}, \quad X_3.\] We can choose a common eigenvector $v \in \ker(X_1)$ such that 
 \[vX_2^{\ell_3}=\psi_1v, \quad vW^{m_3}=\psi_2v, \quad vX_3=\psi_3v,\] where $ \ell_3, m_3 \in \mathbb{N} $ are given by  
$$
  \ell_3 := \ord(rs), \quad \text{and}\quad   m_3 := 
    \begin{cases}
        \ord(r^{-2}s^2), & \text{if } \ord(s^2) \mid \ord(r^{-2}s^2) \ord(rs) , \\[5pt]
        2\ord(r^{-2}s^2), & \text{otherwise}.
    \end{cases}
$$ Sinxe $\beta \neq 0$, we must have $\psi_1 \neq 0$ and $\psi_3 \neq 0$. 
 If $\psi_2 \neq 0$, then all the vectors $vX_2^{a}{W}^{b}$ with $0 \leq a \leq \ell_3-1$ and $0 \leq b \leq m_3-1$ are nonzero.\\
If $\psi_2=0$, let $p$ be the smallest integer with $1 \leq p \leq m_3$ such that $v{W}^{p-1}\neq 0$ and $v{W}^p=0$. In this case, we claim that $p=\ord(r^{-2}s^{2})$. Indeed, consider $\left(v{W}^p\right)X_1=0.$ This gives 
\begin{align*}
   0&=v\left((rs)^{-2p}X_1{W}^p-(rs)^{-2}(1-r^{-1}s)\displaystyle\frac{1-(r^{-1}s)^{2p}}{1-r^{-1}s}{W}^{p-1}X_2^{2}\right)\\
   &=-(rs)^{-2}(1-r^{-1}s)\displaystyle\frac{1-(r^{-1}s)^{2p}}{1-r^{-1}s}\left(v{W}^{p-1}X_2^{2}\right). 
\end{align*}
Note $vW^{p-1}\neq 0$ and $X_2$ is invertible. So we must have $1-(r^{-1}s)^{2p}=0$ i.e., $\ord(r^2s^{-2})\mid p$. By the choice of $m_3$, we see that $p=\ord(r^{-2}s^2)$ or $p=2\ord(r^{-2}s^2)$. We claim that $p=\ord(r^{-2}s^2)$.

\noindent Note if $\ord(s^2) \mid \ord(r^{-2}s^2) \ord(rs)$ then we must have $p=\ord(r^{-2}s^2)$. Let $\operatorname{ord(s^2)}\nmid \ord(r^{-2}s^2) \ord(rs)$. Assume if $p =2 \ord(r^{-2}s^2)$, then the $\mathbb{K}$ linear span \[S:=\big<vX_{2}^a{W}^b:0 \leq a \leq \ell_3-1, \ord(r^2s^{-2})\leq b \leq 2\ord(r^2s^{-2})-1\big>\] is a nonzero submodule of $\mathcal{N}$. Since $\mathcal{N}$ is simple, therefore $S=\mathcal{N}$. In particular $v\in S$, so $v{W}^{\ord(r^2s^{-2})}=0$, which is a contradiction. Thus we define $$
  \ell_3 := \ord(rs), \quad \text{and}\quad   m_3 := 
    \begin{cases}
        \ord(r^{-2}s^2), & \text{if } \ord(s^2) \mid \ord(r^{-2}s^2) \ord(rs)\ \text{or} \ \psi_2=0 , \\[5pt]
        2\ord(r^{-2}s^2), & \text{otherwise}.
    \end{cases}$$
    and the above argument ensures that all the vectors $vX_2^{a}{W}^{b}$ with $0 \leq a \leq \ell_3-1$ and $0 \leq b \leq m_3-1$ are nonzero. Let us choose
\[\nu_1:=\psi_1^{\frac{1}{\ell_3}},\quad  \nu_2:=\psi_3,\quad \nu_3:=\psi_2,\] so that 
$\nu=(\nu_1,\nu_2,\nu_3)\in {(\mathbb{C}^*)}^2 \times \mathbb{C}$.  Now define a $\mathbb{K}$-linear map \[\Phi_4:\mathcal{M}(\nu)\rightarrow \mathcal{N}\] by specifying the image of the basis vectors of $\mathcal{M}(\nu)$ as follows:
\[\Phi_4\left(e(a,b)\right)=\nu_1^{-a}vX_{2}^a{W}^b,\ \ 0\leq a\leq \ell_3-1, 0 \leq b \leq m_3-1.\] We can easily verify that $\Phi_4$ is a nonzero $U_{r,s}^+(B_2)$-module homomorphism. In this verification the following calculations will be useful
\begin{align*}
    \left(vX_2^aW^b\right)X_1=\begin{cases}
        s^{-4(b-1)}(rs)^{-2} (1 - r^{-1}s)  
           \displaystyle \frac{1 - (r^{-1}s)^{2b}}{1 - (r^{-1}s)^2}(vX_2^{a+2}W^{b-1}), &  b\neq 0,\\
        0, & b=0.
    \end{cases}
\end{align*}
Also for any element $w \in \mathcal{N}$, we have 
\[wX_4:=wX_2^{-1}\frac{W-X_3}{s^{-2}-r^{-1}s^{-1}}\] Thus by Schur's lemma, $\Phi_4$ must be a $U_{r,s}^+(B_2)$-module isomorphism.\\
\textbf{Case II:} Let $\beta=0$. Then $X_2$ acts nilpotently on $\mathcal{N}$. Define 
\[\ell_4:= \operatorname{lcm} \left(\ord(rs), \ord(rs^{-1})\right).\]
Since the $\ker(X_1) \cap \ker (X_2)$ is invariant under the commuting operators \[X_4^{\ell_4}, \quad X_3.\] So we can choose a common eigenvector $v \in \ker(X_1) \cap \ker (X_2)$ such that 
\[vX_4^{\ell_4}=\phi_1 v, \quad vX_3=\phi_2 v,\] where $\phi_1, \phi_2 \in \mathbb{C}$. As $\mathcal{N}$ is $X_3$-torsionfree, $\phi_2 \neq 0$.\\
If $\phi_1 \neq 0$, then all the vectors $vX_4^a$ with $0 \leq a \leq \ell_4-1$ are nonzero. 

\noindent If $\phi_1 = 0$, let $p$ be the smallest integer with $1 \leq p \leq \ell_4$ such that $vX_4^{p-1}\neq 0$ and $vX_4^p=0$. In this case, we claim $p=\ord(rs^{-1})$. Indeed, consider $\left(vX_4^p\right)X_2=0$. This gives 
\begin{align*}
   0&=v \left(s^{-2p}X_2X_4^p+(rs)^{-(p-1)}\displaystyle\frac{1-(rs^{-1})^p}{1-rs^{-1}}X_3X_4^{p-1}\right)
   =\phi_2(rs)^{-(p-1)}\displaystyle\frac{1-(rs^{-1})^p}{1-rs^{-1}}\left(vX_4^{p-1}\right).
\end{align*}
Hence we must have $1-(rs^{-1})^p=0$ i.e., $\ord(rs^{-1})\mid p$. We claim that $p=\ord(rs^{-1})$. If not, let $p=k\ord(rs^{-1})$ where $2 \leq k \leq \displaystyle{\ell_4}/{\ord(rs^{-1})}$. Then the $\mathbb{K}$-linear span
\[S:=\big<vX_4^a: \ord(rs^{-1}) \leq a \leq k\ord(rs^{-1})-1\big>\]
 is a nonzero submodule of $\mathcal{N}$. Since $\mathcal{N}$ is simple, $S=\mathcal{N}$. In particular $v \in S$, so $vX_4^{(k-1)\ord(rs^{-1})}=0$, which is a contradiction. Hence $p=\ord(rs^{-1})$. Thus we define 
$$
  \ell_4 := 
    \begin{cases}
        \operatorname{lcm}(\ord(rs),\ord(rs^{-1})), & \text{if } \phi_1\neq 0 , \\[5pt]
        \ord(rs^{-1}), & \text{if } \phi_1=0,
    \end{cases}
$$
and the above argument ensures that all the vectors $vX_4^{a}$ with $0 \leq a \leq \ell_4$ are nonzero. Let us choose
\[\xi_1:=\phi_2,\quad \xi_2:=\phi_1,\] so that
$\xi=(\xi_1,\xi_2)\in {(\mathbb{C}^*)} \times \mathbb{C}$.  Now define a $\mathbb{K}$-linear map \[\Phi_5:\mathcal{M}(\xi)\rightarrow \mathcal{N}\] by specifying the image of the basis vectors of $\mathcal{M}(\xi)$ as follows:
\[\Phi_5\left(e(a)\right)=vX_{4}^a,\ \ 0\leq a\leq \ell_4-1.\] We can easily verify that $\Phi_5$ is a nonzero $U_{r,s}^+(B_2)$-module homomorphism. In this verification the following calculations will be useful
\begin{align*}
    \left(vX_4^a\right)X_1&=\begin{cases}
        \displaystyle -r^{-2(a-1)}\frac{1 - (rs^{-1})^{a-1}}{1 - rs^{-1}}\frac{1 - (rs^{-1})^a}{1 - (rs^{-1})^2}\phi_2\left(vX_4^{a-2}\right), & a \neq 0,1, \\[5pt] 
        0, & a = 0,1.
    \end{cases}\\
    \left(vX_4^a\right)X_2&=\begin{cases}
        \displaystyle (rs)^{-(a-1)}\frac{1 - (rs^{-1})^{a}}{1 - rs^{-1}}\phi_2\left(vX_4^{a-1}\right), & a \neq 0, \\[10pt]
        0, & a = 0.
    \end{cases}
\end{align*} Thus by Schur's lemma, $\Phi_5$ must be a $U_{r,s}^+(B_2)$-module isomorphism. \\
We conclude the section by stating the following theorem.
\begin{theorem}\label{m4}
    Let $\mathcal{N}$ be a simple $X_1$-torsion $U_{r,s}^+(B_2)$-module. Then $\mathcal{N}$ is isomorphic to one of the following modules
    \begin{enumerate}[label= \((\arabic*)\)]
        \item $ \mathcal{M}(\nu) $ for $ \nu = (\nu_1, \nu_2, \nu_3) \in (\mathbb{C}^*)^2 \times \mathbb{C} $ if $\mathcal{N}$ is $X_2$-torsionfree.
        \item $ \mathcal{M}(\xi) $ for $ \xi = (\xi_1, \xi_2) \in \mathbb{C}^*\times \mathbb{C} $ if $\mathcal{N}$ is $X_2$-torsion. 
    \end{enumerate}
\end{theorem}
\section{\texorpdfstring{Isomorphism Classes of simple $U_{r,s}^+(B_2)$-module}{}}

In this section, we describe the isomorphism classes of all finite dimensional simple $ U^+_{r,s}(B_2)$-modules. Using the classification results of previous sections for both $ X_1 $-torsionfree and $ X_1 $-torsion modules, we determine the parameters that distinguish non-isomorphic modules. We identify conditions under which two simple modules, possibly arising from same constructions, are isomorphic. The analysis shows that the parameter spaces involved are disjoint in most cases, ensuring a clear separation of module types. As a consequence, we confirm that our classification is complete and each simple module corresponds uniquely to a parameter tuple up to suitable equivalence.

\begin{thm}\label{1stIso}
    Let $\lambda=(\lambda_1,\lambda_2,\lambda_3,\lambda_4)$, $\lambda'=(\lambda'_1,\lambda'_2,\lambda'_3,\lambda'_4) \in(\mathbb{C}^*)^3 \times \mathbb{C}$. Then $M(\lambda)$ is isomorphic to $M(\lambda')$ as $U_{r,s}^+(B_2)$-module if and only if 
    \begin{align*}
    \lambda_1^{\ell_1}=s^{-2v\ell_1}&(\lambda_1')^{\ell_1}, \ \lambda_2^{m_1}=s^{2um_1}(\lambda_2')^{m_1}, \ \lambda_3=(rs)^u(rs^{-1})^v\lambda_3'\\
    \lambda_4&=(rs)^{2u}\left(\lambda_4'-s^2(r^2-s^2)\displaystyle\frac{1-(rs^{-1})^v}{1-rs^{-1}}\lambda_3'\right)
\end{align*} holds for some $0 \leq u \leq \ell_1-1$ and $0 \leq v \leq m_1-1$.
\end{thm}
\begin{proof}
Suppose that there exists an isomorphism $\phi: M(\lambda) \to M(\lambda')$. Assume that
\[
\phi(e(0,0)) = \sum J_{u,v} \, e(u,v),
\]
where the sum is finite and at least one coefficient $J_{u,v}$ lies in $\mathbb{C}^*$.

\noindent We claim that only one coefficient is nonzero others are zero. Assume, for contradiction, that there exist two distinct pairs $(u,v) \neq (u',v')$ such that both $J_{u,v}$ and $J_{u',v'}$ are nonzero. We have
\begin{align*}
   \phi(e{(0,0)} X_3 X_1)& = \phi(e{(0,0)}) X_3 X_1 \\
   \phi(e{(0,0)} \tilde{W}X_2) &= \phi(e{(0,0)}) \tilde{W}X_2, \\ \phi(e{(0,0)}X_2^{m_1})& = \phi(e{(0,0)}) X_2^{m_1}. 
\end{align*}
\noindent By comparing the coefficients of the basis elements on both sides of these identities, we derive the following
\begin{align*}
    (rs)^{2u}(rs^{-1})^v \lambda_3'&=(rs)^{2u'}(rs^{-1})^{v'} \lambda_3'\\
        (rs)^{2u}\left(\lambda_4'-s^2(r^2-s^2)\displaystyle\frac{1-(rs^{-1})^v}{1-rs^{-1}}\lambda_3'\right)&=(rs)^{2u'}\left(\lambda_4'-s^2(r^2-s^2)\displaystyle\frac{1-(rs^{-1})^{v'}}{1-rs^{-1}}\lambda_3'\right)\\
        s^{2um_1}(\lambda_2')^{m_1}&=s^{2u'm_1}(\lambda_2')^{m_1}.
\end{align*}
In particular, this leads to the congruences $u \equiv u' (\mod \ell_1)$ and $v \equiv v' (\mod m_1)$, which contradict our earlier assumption that $(u,v) \neq (u',v')$.

\noindent Hence, we conclude that
\[
\phi(e{(0,0)}) = J_{u,v} \, e{(u,v)}
\]
for some unique pair $(u,v)$ and some scalar $J_{u,v} \in \mathbb{C}^*$.

\noindent This observation allows us to deduce relations among the scalars that appear in the module action. Let us consider the actions of $X_1^{\ell_1}$, $X_2^{m_1}$, $X_3X_1$, $\tilde{W}X_2$ on the basis elements under the isomorphism $\phi$ to obtain
\begin{align*}
    \lambda_1^{\ell_1}=s^{-2v\ell_1}&(\lambda_1')^{\ell_1}, \ \lambda_2^{m_1}=s^{2um_1}(\lambda_2')^{m_1}, \ \lambda_3=(rs)^u(rs^{-1})^v\lambda_3'\\
    \lambda_4&=(rs)^{2u}\left(\lambda_4'-s^2(r^2-s^2)\displaystyle\frac{1-(rs^{-1})^v}{1-rs^{-1}}\lambda_3'\right)
\end{align*}
Conversely, assume that the relations between $\lambda$ and $\lambda'$ hold. Then the map $\tilde{\psi}:M(\lambda)\rightarrow M(\lambda')$ defined by
\[\tilde{\psi}\left(e(a,b)\right)=\begin{cases}
(\lambda_1^{-1}\lambda_1')^a(\lambda_2^{-1}\lambda_2')^bs^{-2av}e(a\oplus u,b+v), & 0 \leq b \leq m_1-v-1\\
(\lambda_1^{-1}\lambda_1')^a(\lambda_2^{-1}\lambda_2')^bs^{-2av}s^{2(a \oplus u)m_1}e(a\oplus u,b+v), & m_1-v \leq b \leq m_1-1   
\end{cases}\]
One can verify that $\tilde{\psi}$ is a $U_{r,s}^+(B_2)$-module isomorphism.
\end{proof}
\begin{theorem}\label{2ndIso}
    Let $\mu=(\mu_1,\mu_2,\mu_3)$ $\mu'=(\mu'_1,\mu'_2,\mu'_3)\in (\mathbb{C}^*)^2 \times \mathbb{C}$. Then $M(\mu)$ is isomorphic to $M(\mu')$ as $U_{r,s}^+(B_2)$-module if and only if 
\[\mu_1^{\ell_2}=r^{-2v\ell_2}(\mu_1')^{\ell_2},\ \mu_2=(rs)^{2u+v}\mu_2',\ \mu_3=r^{2um_2}\mu_3'\] for some $u,v$ such that $0 \leq u \leq \ell_2-1$ and $v\in\{0,\ord(rs^{-1})\}$.
\end{theorem}
\begin{proof}
    Using similar argument as in Theorem \ref{1stIso} we can show that any isomorphism $\tilde{\psi}: M(\mu) \rightarrow M(\mu')$ is characterized by $e(0,0)$ and $\tilde{\psi}\left(e(0,0)\right)=J_{u,v}e(u,v)$ for some $J_{u,v}\in \mathbb{C}^*$. We can simplify the equality $\tilde{\psi}(e(0,0)X_2)=\tilde{\psi}(e(0,0))X_2$ to obtain $0=J_{u,v}e(u,v)X_2$ which imply either $v=0$ or 
    \[0=(rs)^{2u}s^2(r^2-s^2)\displaystyle\frac{1-(r^{-1}s)^v}{1-r^{-1}s}\mu_2'\] i.e.,  $v=\ord(r^{-1}s)$. Then the actions of $X_1^{\ell_2}, \ X_3X_1$ and $\tilde{W}^{m_2}$ under $\tilde{\psi}$ provide the required relations between $\mu$ and $\mu'$. 
    \par Conversely, assume that the relations between $\mu$ and $\mu'$ hold. Then the map $\tilde{\psi}:M(\mu)\rightarrow M(\mu')$ defined by
\[\tilde{\psi}\left(e(a,b)\right)=\begin{cases}
(\mu_1^{-1}\mu_1')^a r^{-2bv}e(a\oplus u,b+v), & 0 \leq b \leq m_2-v-1\\
(\mu_1^{-1}\mu_1')^a r^{-2bv}r^{2(a \oplus u)m_2}e(a\oplus u,b+v), & m_2-v \leq b \leq m_1-1   
\end{cases}\]
One can verify that $\tilde{\psi}$ is a $U_{r,s}^+(B_2)$-module isomorphism.
\end{proof}
\begin{thm}
 Let   $ \epsilon = (\epsilon_1, \epsilon_2, \epsilon_3)$, $  \epsilon' = (\epsilon_1', \epsilon_2', \epsilon_3')  \in (\mathbb{C}^*)^3$. Then $M(\epsilon)$ is isomorphic to $M(\epsilon')$ as $U_{r,s}^+(B_2)$-module if and only if 
 \[\epsilon_1^{\ell}=(\epsilon_1')^{\ell}, \ \epsilon_2=(rs)^u(\epsilon_2)', \ \epsilon_3=(rs^{-1})^u\epsilon_3'\] for some $u$ such that $0 \leq u \leq \ell-1$.
\end{thm}
\begin{proof}
  The proof is parallel to the proof of Theorem \ref{1stIso}.
\end{proof}
\begin{theorem}
    Let $ \nu = (\nu_1, \nu_2, \nu_3)$, $ \nu' = (\nu_1', \nu_2', \nu_3')\in (\mathbb{C}^*)^2 \times \mathbb{C} $. Then $M(\nu)$ is isomorphic to $M(\nu')$ as $U_{r,s}^+(B_2)$-module if and only if 
    \[\nu_1^{\ell_3}=s^{-2v\ell_3}(\nu_1')^{\ell_3}, \ \nu_2=(rs)^v\nu_2', \ \nu_3=s^{2um_3}\nu_3'\] for some $u$ such that $0 \leq u \leq \ell_3-1$ and $v$ such that either $v=0$ or $v=\ord(r^{-2}s^2)$.
\end{theorem}
\begin{proof}
   The proof is parallel to the proof of Theorem \ref{2ndIso}.
\end{proof}
\begin{theorem}\label{5thIso}
    Let $ \xi = (\xi_1, \xi_2)$, $ \xi' = (\xi_1', \xi_2')\in \mathbb{C}^*\times \mathbb{C} $. Then $M(\xi)$ is isomorphic to $M(\xi')$ as $U_{r,s}^+(B_2)$-module if and only if
    \[\xi_1=(rs)^{-u}\xi_1', \ \xi_2=\xi_2'\] holds for some $u$ with $0 \leq u \leq \ell_4-1$ and $u$ is a multiple of $\ord(rs^{-1})$.
\end{theorem}
\begin{proof}
    The proof is parallel to the proof of Theorem \ref{1stIso}.
\end{proof}
\section{A Class of Indecomposable Modules}

In this final section, as an example, we give a class of indecomposable $U_{r,s}^+(B_2)$ modules that are not simple.

Let $ B \subseteq U_{r,s}^+(B_2) $ be the subalgebra generated by $ X_1, X_2, X_3 $. For $ \lambda \in \mathbb{C}^* $, define a 1-dimensional $ B $-module $ \mathbb{C}_\lambda = \mathbb{C}v $, where the action is given by:
    $$vX_1=vX_2=0,\quad vX_3=\lambda v.$$
    Now, consider the induced module:
    $$M(\lambda) := \mathbb{C}_\lambda \otimes_B U_{r,s}^+(B_2).$$
    Since $ U_{r,s}^+(B_2) $ has a PBW basis $ \{ X_1^a X_2^b X_3^c X_4^d\mid a,b,c,d\in \Z_{\geq 0} \} $, it follows that $\{f(k)=v\otimes X_4^k\mid k\in \Z_{\geq 0} \}$ is a basis of $M(\lambda)$. Using the relations of Equation (\ref{Eq2.1}) and the identities of Lemma \ref{lemma2.2}, we obtain the following explicit actions of $X_1,X_2,X_3$ and $X_4$ on the basis  elements of $M(\lambda)$:
    $$f(k)X_1=\begin{cases}
       \displaystyle -\lambda r^{-2k}r^2\frac{\left((rs^{-1})^k-1\right)\left((rs^{-1})^{k-1}-1\right)}{\left((rs^{-1})^2-1\right)\left(rs^{-1}-1\right)}f(k-2), & \text{for } k\geq 2\\
       0, & \text{for } k=0,1
       \end{cases}$$
$$f(k)X_2=\begin{cases}
    \displaystyle\lambda (rs)^{-(k-1)}\frac{1-(rs^{-1})^k}{1-rs^{-1}}f(k-1), &\text{for } k\geq 1\\
    0,& \text{for } k= 0
    \end{cases}$$
    $$ f(k)X_3= \lambda(rs)^{-k}f(k), \quad f(k)X_4=f(k+1) \quad \text{for } k\geq 0.$$
    Let $ {m} = \mathrm{ord}(rs^{-1}) $. For each $k\geq 0$ define the subspaces:
$$
M_{k,m} := \mathrm{span}_\mathbb{C} \{ f(km + i) \mid i \geq 0 \}.
$$
One can verify that each $ M_{k,m} $ is a submodule of $ M(\lambda) $. We have the following descending chain of submodules of $ M(\lambda) $:
$$M_{0,m}\supset M_{1,m}\supset M_{2,m}\supset\cdots.$$
Let $Q_{k,m} := M_{0,m} / M_{k,m}$. Clearly $\mathrm{dim}\,Q_{k,m}=km$. We have the following ascending chain of quotients of $ M(\lambda) $: 
$$Q_{0,m}\subset Q_{1,m}\subset Q_{2,m}\cdots.$$ Also observe that for $k \geq 1$
\[Q_{k,m}= \operatorname{Span}\{\overline{f(a)}=f(a) \operatorname{mod} M_{k,m}: 0 \leq a \leq km-1\}\] On each $Q_{k,m}$ the above actions on $M_{k,m}$ induces $km$-dimensional $U_{r,s}^+(B_2)$-module structure:
\begin{align*}
    \overline{f(a)}X_1&=\begin{cases}
       \displaystyle -\lambda r^{-2a}r^2\frac{\left((rs^{-1})^a-1\right)\left((rs^{-1})^{a-1}-1\right)}{\left((rs^{-1})^2-1\right)\left(rs^{-1}-1\right)}\overline{f(a-2)}, & \text{for } a\geq 2\\
       0, & \text{for } a=0,1
       \end{cases}\\
\overline{f(a)}X_2&=\begin{cases}
    \displaystyle\lambda (rs)^{-(a-1)}\frac{1-(rs^{-1})^a}{1-rs^{-1}}\overline{f(a-1)}, &\text{for } a\geq 1\\
    0,& \text{for } a= 0
    \end{cases}\\
   \overline{f(a)}X_3&= \lambda(rs)^{-a}\overline{f(a)}, \quad \overline{f(a)}X_4=\begin{cases}
       \overline{f(a+1)},& \text{for } a\neq km-1.\\
       0,&a=km-1
   \end{cases}
\end{align*}
The following result provides all submodules of $Q_{k,m}$.
\begin{theorem}
    The $U_{r,s}^+(B_2)$-submodules of $M(\lambda)$ containing $M_{k,m}$ are of the form $M_{r,m}$ where $0 \leq r \leq k$.
\end{theorem}
\begin{proof}
    Let $W$ be a submodule of $M(\lambda)$ containing $M_{k,m}$. If $W=M_{k,m}$ then we are done. If not then $W$ contains an element $w$ of the form 
    \[w=\sum_{p=0}^{km-1}c_pf(p)\] where $c_p \in \mathbb{C}$ with at least one $c_p \neq 0$. Let $i:=\operatorname{min}\{p:c_p\neq 0\}$. Then $wX_4^{km-1-i}\in W$. Now 
    \[wX_4^{km-1-i}=c_i f(km-1)+w'\] where $w'\in M_{k,m}$. Hence $f(km-1)\in W$. Finally with the action of $X_2$ on $f(km-1)$, we obtain $M_{k-1,m}\subset W$. Now, if $W = M_{k-1,m}$, then we are done. Otherwise, continuing with the above argument sequentially, one can obtain the desired result.
\end{proof}
Hence we have the following theorem.
\begin{theorem}
    \begin{enumerate}[label= \((\arabic*)\)]
        \item The $U_{r,s}^+(B_2)$-module $Q_{1,m}$ is a simple module which is isomorphic to the simple module $M(\xi)$ with $\xi_1=\lambda$ and $\xi_2=0$.
        \item Each $Q_{k,m}$ for $k \geq 2$ is an indecomposable but not simple $U_{r,s}^+(B_2)$-module.
    \end{enumerate}
\end{theorem}

\begin{remark}
This construction opens up several directions for further investigation. One natural avenue is to study the extension groups between the modules $ Q_{k,m} $, aiming to understand their place within the broader category of finite dimensional $ U_{r,s}^+(B_2)$-modules. Another promising direction is to explore connections with the theory of crystal bases, especially in the root of unity setting, where categorifications or deformations of these indecomposable modules might reveal deeper structural and combinatorial patterns aligned with the crystal graphs of quantized enveloping algebras.

\end{remark}

\section*{Acknowledgment} The research of the first author is financially supported by the Institute Postdoctoral Fellowship, IIT Kanpur. The second author would like to thank IIT Kanpur for the support of the FARE (Fellowship for Academic and Research Excellence) fellowship.
\section*{Data Availability Statement}
Data sharing does not apply to this article as no datasets were generated or analyzed during the current study.
\section*{Appendix} 
Here we consider the special case when $r =\pm s$. Under this assumption, the defining relations of $U_{r,s}^+(B_2)$ simplify, leading to a clearer structure that allows for the construction of explicit infinite dimensional simple modules and hence proves that the algebra is not PI.\\
When $r=s$, the relations between the generators become the following
\begin{align*}
    X_1X_2=r^2X_2X_1, \ X_1X_3=r^4X_3X_1&, \ X_2X_3=r^2X_3X_2, \ X_4X_3=r^{-2}X_3X_4\\
    X_2X_4=r^2X_4X_2-r^2X_3, \ & X_1X_4=r^2X_4X_1+X_2.
\end{align*} Assume $r^2=q$. Then we have the following lemma.
\begin{lemma}
    We have the following identities.
    \begin{enumerate}[label= \((\arabic*)\)]
        \item $X_2^kX_4=q^kX_4X_2^k-kq^kX_3X_2^{k-1}$.
        \item $X_4^kX_2=q^{-k}X_2X_4^k+kq^{-(k-1)}X_3X_4^{k-1}$.
        \item $X_1X_4^k=q^kX_4^kX_1+kq^{k-1}X_4^{k-1}X_2-\displaystyle \frac{k(k-1)}{2}q^{k-1}X_4^{k-2}X_3$.
        \item $X_1^kX_4=q^kX_4X_1^k+kq^{k-1}X_2X_1^{k-1}$.
    \end{enumerate}
\end{lemma}
Then the module $M(\lambda)$ defined in the previous section becomes a simple module over $U_{r}^+(B_2)$. Indeed the actions are defined as follows
\begin{align*}
    f(k)X_1&=\begin{cases}
        -\displaystyle\frac{k(k-1)}{2}q^{-(k-1)}\lambda f(k-2), & k \geq 2,\\
        0, & \text{otherwise,}
        \end{cases}\\
        f(k)X_2&=\begin{cases}
            kq^{-(k-1)}\lambda f(k-1),& k \geq 1,\\
            0,& k=0,
        \end{cases} \\
        f(k)X_3&=q^{-k}\lambda f(k),\quad
        f(k)X_4=f(k+1).
\end{align*}
 Since $U_{r}^+(B_2)$ has an infinite dimensional simple module, it is not a PI algebra.\\
Similarly, if $r=-s$, we can show that $U_{r,-r}^+(B_2)$ is not a PI algebra.
\bibliographystyle{plain}
\bibliography{reference}

\end{document}